\documentclass[11pt, reqno]{amsart}

\usepackage{amssymb, amscd, latexsym, mathrsfs}
\usepackage{amsmath}
\usepackage{amsthm}
\usepackage{float}
\usepackage{pgf, tikz}
\usepackage{algorithm, algorithmic}
\usepackage[bookmarksnumbered,colorlinks, plainpages]{hyperref}
\hypersetup{
pdftitle={Research Paper},
pdfauthor={Ali Akbar Yazdan Pour},
pdfsubject={Positive Matching Decomposition},
colorlinks,
linkcolor=blue,
citecolor=magenta,
anchorcolor=red,
bookmarksopen,
urlcolor=red,
filecolor=red,
}
\theoremstyle{plain}
\newtheorem{theorem}{Theorem}[section]
\newtheorem{proposition}[theorem]{Proposition}
\newtheorem{lemma}[theorem]{Lemma}
\newtheorem{corollary}[theorem]{Corollary}

\theoremstyle{definition}
\newtheorem*{definition}{Definition}
\newtheorem*{notation}{Notation}
\newtheorem*{example}{Example}

\theoremstyle{remark}
\newtheorem{remark}{Remark}

\newcommand{\KK}{\mathbb{K}}
\newcommand{\RR}{\mathbb{R}}
\newcommand{\paths}{\mathrm{Path}}
\newcommand{\edges}{\mathrm{Edge}}
\newcommand{\End}{\mathrm{End}}
\newcommand{\init}{\mathrm{in}}
\newcommand{\term}{\mathrm{ter}}
\newcommand{\Hilb}{\mathrm{Hilb}}
\renewcommand{\P}{\mathscr{P}}
\newcommand{\Q}{\mathscr{Q}}

\newcommand{\calP}{\mathcal{P}}
\newcommand{\supp}{\mathrm{supp}}
\newcommand{\Supp}{\mathrm{Supp}}
\newcommand{\parn}[1]{\left(#1\right)}
\renewcommand{\S}{\mathcal{S}}

\begin{document}
\title[Gr\"obner basis of LSS-ideal associated to trees]{Gr\"obner basis and Krull dimension of Lov\'asz-Saks-Sherijver ideal associated to a tree}
\author[M. Farrokhi D. G.]{Mohammad Farrokhi D. G.}
\email{m.farrokhi.d.g@gmail.com,\ farrokhi@iasbs.ac.ir}
\address{ Research Center for Basic Sciences and Modern Technologies (RBST), Institute for Advanced Studies in Basic Sciences (IASBS), 
Zanjan 45137-66731, Iran}

\author[A. A. Yazdan Pour]{{Ali Akbar} {Yazdan Pour}}
\email{yazdan@iasbs.ac.ir}
\address{Department of Mathematics, Institute for Advanced Studies in Basic Sciences (IASBS), Zanjan 45137-66731, Iran}
			
	
\begin{abstract}
Let $\KK$ be a field and $n$ be a positive integer. Let $\Gamma =([n], E)$ be a simple graph, where $[n]=\{1,\ldots, n\}$. If $S=\KK[x_1, \ldots, x_n, y_1, \ldots, y_n]$ is a polynomial ring, then the graded ideal 
\[ L_{\Gamma}^\KK(2) = \left( x_{i}x_{j} + y_{i}y_{j} \colon \quad \{i, j\} \in E({\Gamma})\right) \subset S,\]
is called the Lov\'{a}sz-Saks-Schrijver ideal, LSS-ideal for short, of $\Gamma$ with respect to $\mathbb{K}$. In the present paper, we compute a Gr\"obner basis of this ideal with respect to lexicographic ordering induced by $x_1>\cdots>x_n>y_1>\cdots>y_n$ when $\Gamma=T$ is a tree. As a result, we show that it is independent of the choice of the ground field $\KK$ and compute the Hilbert series of $L_{T}^\KK(2)$. Finally, we present concrete combinatorial formulas to obtain the Krull dimension of $S/L_{T}^\KK(2)$ as well as lower and upper bounds for Krull dimension.
\end{abstract}

\maketitle
\section{introduction and preliminaries}
Gr\"obner bases of an ideal, as special kinds of generating sets of the ideal, give rise to computational tools to determine many properties of the ideal. It is one of the main practical tools for solving systems of polynomial equations and computing the images of algebraic varieties under projections or rational maps. Gr\"obner bases were introduced in 1965, together with
an algorithm to compute them (Buchberger's algorithm), by Bruno Buchberger in his Ph.D. thesis. This notion has been introduced for ideals in the polynomial rings and extended to various noncommutative algebras such as free non-commutative algebra and exterior algebra \cite{aa-jh-th}.

Let $I$ be an ideal in the polynomial ring $R=\KK[x_1, \ldots,x_n]$. We denote by $\mathrm{Mon}(R)$ the set of all monomials in $R$, that is the set $\{x_1^{\alpha_1} \cdots x_n^{\alpha_n} \colon \; \alpha_1, \ldots, \alpha_n \in \mathbb{N} \cup \{0\} \}$. Let $\leq$ be a monomial ordering on $R$ by which we mean $\leq$ is a total ordering on $\mathrm{Mon}(R)$ such that:
\begin{itemize}
\item[(i)] $1 < u$ for all $1 \neq u \in \mathrm{Mon}(R)$;
\item[(ii)] if $u, v \in \mathrm{Mon}(R)$ and $u < v$, then $uw < vw$ for all $w \in \mathrm{Mon}(R)$.
\end{itemize}
Some well-known monomial orders are lexicographic $\leq_{\mathrm{lex}}$, degree lexicographic $\leq_{\mathrm{dlex}}$, and degree reverse lexicographic $\leq_{\mathrm{revlex}}$ ordering (see \cite{wwa-pl} for detailed descriptions). 

Let $\leq$ be a fixed monomial ordering on $\mathrm{Mon}(R)$ and $\mathbb{N}_\circ$ be the set of non-negative integers. If $f = \sum_{\alpha \in \mathbb{N}_\circ^n} a_\alpha \textbf{x}^\alpha$ is a nonzero polynomial of $R$ with each $a_\alpha\in \KK$, then the initial monomial of $f$ with respect to $\leq$ is the biggest monomial with respect to $\leq$ among the monomials belonging to $\supp(f):=\{\mathbf{x}^\alpha \colon \; a_\alpha\neq 0\}$. Let $\init_\leq(f)$ be the initial monomial of $f$ with respect to $\leq$. If $I$ is a nonzero ideal of $R$, the initial ideal of $I$ with respect to $\leq$ is the monomial ideal of $R$ generated by $\{\init_\leq(f) \colon \; 0\neq f \in I \}$. Let $\init_\leq(I)$ denote the initial ideal of $I$, i.e.
\[\init_\leq(I)= (\init_\leq(f) \colon \; 0\neq f \in I).\]

\begin{definition}
Let $I$ be a nonzero ideal of $R$. A finite set of nonzero polynomials $\{g_1, \ldots , g_s\}$ where $g_i \in I$ is said to be a Gr\"obner basis of $I$ with respect to $\leq$ if the initial ideal $\init_\leq(I)$ of $I$ is generated by the monomials $\init_\leq(g_1), \ldots, \init_\leq(g_s)$.
\end{definition}

Let $\leq$ be a monomial ordering on $R$. The following facts are well-known in the theory of Gr\"obner bases:
\begin{itemize}
\item Gr\"obner bases always exist for any ideal $I$ of $R$. Moreover, any Gr\"obner basis of $I$ can be transformed into a reduced Gr\"obner basis $\{g_1,\ldots,g_m\}$ in the following sense:
\begin{itemize}
\item the coefficient of $\init_\leq(g_i)$ in $g_i$ is $1$, for all $1\leq i\leq m$,
\item If $i\neq j$, then none of the monomials belonging to $\supp(g_j)$ is divisible by $\init_\leq(g_i)$.
\end{itemize}
\item If $\mathcal{G}$ is a Gr\"obner basis of $I$ then so is $\mathcal{G}'$ for any $\mathcal{G} \subseteq \mathcal{G}' \subseteq I$.
\item Any Gr\"obner basis of $I$ is a set of generators for $I$.
\item If $\mathcal{G} = \{g_1, \ldots, g_s\}$ is a Gr\"obner basis of $I$ and if $f_1, \ldots, f_s \in I$ are nonzero polynomials with $\init_\leq(f_i) = \init_\leq(g_i)$, then $\{f_1, \ldots,f_s\}$ is also a Gr\"obner basis of $I$.
\item If $\mathcal{G} = \{g_1, \ldots, g_s\}$ is a Gr\"obner basis of $I$, then every polynomial $f \in R$ has a unique remainder $r$ with respect to $g_1, \ldots, g_s$ (see \cite[Lemma 2.2.3]{jh-th}). In this case, we say that $f$ reduces to $r$ with respect to $g_1, \ldots, g_s$.
\item If $\mathcal{G} = \{g_1, \ldots, g_s\}$ is a Gr\"obner basis of $I$, then a nonzero polynomial $f \in R$ belongs to $I$ if and only if the unique remainder of $f$ with respect to $g_1, \ldots, g_s$ is $0$.
\item If $f, g_1, \ldots, g_s$ are non-zero polynomials in $R$, then division algorithm (see e.g. \cite[Theorem 2.2.1]{jh-th}) guarantees the existence of the following standard form
\[f=f_1g_1+\cdots+f_sg_s+r\]
for $f$, where
\begin{itemize}
\item[(a)]if $r\neq 0$ and $u\in\supp(r)$, then none of the initial monomials $\init_\leq(g_1),\ldots,\init_\leq(g_s)$ divides $u$,
\item[(b)]if $f_i\neq0$, then $\init_\leq(f)\geq\init_\leq(f_ig_i)$.
\end{itemize}
\end{itemize}
Let $\mathcal{G}=\{f_1, \ldots,f_s\}$ be a set of generators of $I$. The Buchberger's criterion provides an efficient way to check whether $\mathcal{G}$ is a Gr\"obner basis of $I$ or not. Given nonzero polynomials $f, g \in R$, let $c_f$ and $c_g$ denote the coefficients of $\init_\leq(f)$ and $\init_\leq(g)$ in $f$ and $g$, respectively. The polynomial
\[S(f, g) = \frac{\mathrm{lcm}(\init_\leq(f), \init_\leq(g))}{c_f \cdot \init_\leq(f)}f - \frac{\mathrm{lcm}(\init_\leq(f), \init_\leq(g))}{c_g \cdot \init_\leq(g)}g\]
is called the $S$-polynomial of $f$ and $g$. We now come to the most important theorem in the theory of Gr\"obner bases.
\begin{theorem}[Buchberger's criterion] \label{Buchberger criterion}
Let $I$ be a nonzero ideal of $R$ and $\mathcal{G} = \{g_1, \ldots, g_s\}$ be a system of generators of $I$. Then $\mathcal{G}$ is a Gr\"obner basis of $I$ if and only if  $S(g_i, g_j)$ reduces to $0$ with respect to $g_1, \ldots,g_s$, for all $i \neq j$.
\end{theorem}
The Buchberger criterion (Theorem \ref{Buchberger criterion}) supplies an algorithm to compute a Gr\"obner basis starting from a system of generators of an ideal. This algorithm is so-called as  Bughberger's algorithm. In general, the complexity of this algorithm is of double exponential order. Finding an explicit and quick algorithm for computing a Gr\"obner basis of some families of ideals is very important in the context of computational commutative algebra. The aim of this paper is to compute explicitly a Gr\"obner basis of LSS-ideals associated to trees.

Let $\KK$ be a field, $n$ be a positive integer, and set $[n]=\{1,\ldots, n\}$. Let $\Gamma =([n], E)$ be a simple graph with vertex set $[n]$ and edge set $E$. In \cite{jh-am-ssm-vw} and later in \cite{ac-vw, ak}, the authors study the graded ideal 
\[ L_{\Gamma}^\KK(d) = \left( x_{i,1}x_{j,1} + \cdots + x_{i,d}x_{j,d} \colon \quad \{i, j\} \in E({\Gamma})\right)\]
in the polynomial ring $\KK[x_{i,k} \colon \; i = 1,\ldots,n,\; k = 1, \ldots, d]$. The ideal $L_{\Gamma}^\KK(d)$ is called the \textit{Lov\'{a}sz-Saks-Schrijver ideal}, \textit{LSS-ideal} for short, of $\Gamma$ with respect to $\KK$. For $d = 1$, the ideal $L_\Gamma^\KK(d)$ is a square-free monomial ideal known as the \textit{edge ideal} of $\Gamma$. The algebraic properties of LSS-ideal, such as being prime, radical, and complete intersection are studied in \cite{ac-vw,jh-am-ssm-vw}. Also, in \cite{jh-am-ssm-vw} the (minimal) primary decomposition of $L_\Gamma^\KK(2)$ has been presented in terms of combinatorial properties of $\Gamma$. Let $S=\KK[x_1,\ldots,x_n,y_1,\ldots,y_n]$. In \cite[Corollary 3.6]{ak} the author shows that $L_\Gamma^\KK(2)$ is a complete intersection if and only if the parity binomial edge ideal $\mathcal{I}_{\Gamma}\subseteq S$ of $\Gamma$, defined as below, is a complete intersection:
\[\mathcal{I}_{\Gamma} = \left( x_ix_j-y_iy_j \colon \quad \{i,j\} \in E(\Gamma) \right).\]
Recall that an ideal $I\subseteq R$ is called \textit{complete intersection} if $I$ is generated by a sequence $a_1,\ldots,a_m$ such that $a_i$ is not a zero-divisor of $R/(a_1,\ldots,a_{i-1})$, for $i=1,\ldots,m$. Note that $L_\Gamma^\KK(2)=\mathcal{I}_G$ if $\mathrm{char}(\KK) = 2$. However, it is shown in \cite{ak} that the graded Betti numbers $\beta_{i,j}(S/L_\Gamma^\KK(2))$ and $\beta_{i,j}(S/\mathcal{I}_{\Gamma})$ coincide in any characteristic. It turns out that the algebraic invariants of the LSS-ideal and parity binomial ideal (such as dimension, depth, regularity, projective dimension, etc.) are the same.

When $\KK=\mathbb{R}$ the vanishing set of the ideal $L_{\Gamma}^\KK(d)$ coincides with the variety of orthogonal representations of complement of $\Gamma$ in $\RR^d$. Orthogonal representations of graphs were introduced by Lov\'{a}sz in 1979 \cite{ll2}, and it is shown that they are intimately related to important combinatorial properties of graphs (see \cite[Chapter 10]{ll1}). An \textit{orthogonal representation} of $\Gamma$ in $\mathbb{R}^d$ assigns to each $i \in [n]$ a vector $u_i \in \mathbb{R}^d$ such that $u^T_i u_j = 0$, whenever $\{i,j\} \in \bar{E}=\binom{[n]}{2} \setminus E$. The variety of orthogonal representations was studied in the works of Lov\'asz, Saks, and Schrijver \cite{ll-ms-as-89, ll-ms-as-00}. The reference \cite{ll2} is a nice reference for a comprehensive treatment of orthogonal representations, their varieties, and their relations to important combinatorial properties of graphs.

We find a Gr\"obner basis of $L_{T}^\KK(2)$ with respect to lexicographic ordering induced by $x_1>\cdots>x_n>y_1>\cdots>y_n$ when $T$ is a tree, that is a connected graph without any cycle. Theorem~\ref{Groebner basis of LSS_T(2)} shows that Gr\"obner basis of $L_T^{\KK}(2)$ can be described by paths of even and odd length in $T$. We show in Corollary~\ref{Groebner basis of LSS_T(2): T has descending labeling} that Gr\"obner basis has a simple form if we restrict ourselves to a monomial ordering induced by an ascending labeling on vertices of $T$. It turns out that the Gr\"obner basis of $L_{T}^\KK(2)$ (and hence its Hilbert series and Krull dimension) with respect to this ordering is independent of the choice of the ground field $\KK$. Also, in Theorem~\ref{Hilbert series} and Proposition~\ref{dim(S/I): Explicit formula}, we present  explicit formulas for the Krull dimension of $S/L_{T}^\KK(2)$ in terms of the combinatorics of $T$. Finally, in Corollary~\ref{Lower and upper bounds for dim(S/I)}, we show that the Krull dimension of $S/L_{T}^\KK(2)$ is bounded below by $n+1$ and bounded above by $n-1+p(T)$ where $p(T)$ is the number of pendants of $T$. Recall that a pendant vertex in $T$ is a vertex which has a unique neighbour in $T$. For the graph theoretical concepts in this paper, the reader may refer to \cite{jab-usrm}.
\section{Gr\"obner basis of LSS-ideal of trees}
In this section, we compute a Gr\"obner basis of LSS-ideal of a given tree $T$ with respect to lexicographic ordering induced by
\[x_1>\cdots>x_n>y_1>\cdots>y_n.\]
It turns out that the Gr\"obner basis is characterized by even and odd paths in $T$. 

In what follows, any set of vertices of a path $P\colon v_1,\ldots,v_{2m+1}$ of even length including vertices of odd indices except the end vertices is called an \textit{odd subset} of $P$. For any binomial $w=u\pm v$ with $u,v\in\mathrm{Mon}(S)$ and $u>v$ let $\init_\leq(w):=u$ and $\term_\leq(w):=v$.
\begin{theorem}\label{Groebner basis of LSS_T(2)}
Let $T$ be a tree with vertex set $[n]$ and $I:=L_T^{\KK}(2)$ be the LSS-ideal of $T$. Put
\begin{itemize}
\item[\rm (i)]for every path $P\colon v_1,v_2,\ldots,v_{2m}$ of $T$ of odd length
\[g_P:=(x_{v_1}x_{v_{2m}}+y_{v_1}y_{v_{2m}})z_P,\]
where $z_P=y_{v_3}\cdots y_{v_{2m-1}}$ and $v_1<v_{2m}$;
\item[\rm (ii)]for any path $P\colon v_1,v_2,\ldots,v_{2m+1}$ of $T$ of even length and any odd subset $\mathcal{O}$ of $P$,
\[g'_{P,\mathcal{O}}:=(x_{v_1}y_{v_{2m+1}}-x_{v_{2m+1}}y_{v_1})z'_{P,\mathcal{O}},\]
where $z'_{P,\mathcal{O}}=(\prod_{v\in\mathcal{O}}x_v)\cdot(\prod_{v\in V(P)\setminus(\mathcal{O}\cup\{v_1,v_{2m+1}\})}y_v)$ and $v_1<v_{2m+1}$.
\end{itemize}
Then 
\begin{align*}
\mathcal{G}:=&\{g_P \colon \quad P\emph{ is a path of odd length in }T\}\\
&\cup\{g'_{P,\mathcal{O}} \colon \quad P\emph{ is a path of even length in }T\emph{ and }\mathcal{O}\emph{ is an odd subset of }P\}
\end{align*}
is a Gr\"{o}bner basis of $I$ with respect to the lexicographic ordering induced by $x_1>\cdots>x_n>y_1>\cdots>y_n$. In particular, $I$ is a radical ideal.
\end{theorem}
\begin{proof}
First we show that $\mathcal{G} \subseteq I$. To do this, we show that every polynomial in parts (i)--(ii) belongs to $I$. 

(i) Let $P\colon v_1,\ldots,v_{2m}$ be a path of odd length in $T$, and assume without loss of generality that $v_i=i$ for $i=1,\ldots,2m$. Then
\begin{multline*}
g_P=\frac{z}{y_1y_2}(x_1x_2+y_1y_2)\\
+\sum_{i=1}^{m-1}\left(\frac{x_1x_{2i+2}z}{y_1y_{2i}y_{2i+1}y_{2i+2}}(x_{2i}x_{2i+1}+y_{2i}y_{2i+1})-\frac{x_1x_{2i}z}{y_1y_{2i}y_{2i+1}y_{2i+2}}(x_{2i+1}x_{2i+2}+y_{2i+1}y_{2i+2})\right)
\end{multline*}
belongs to $I$, where $z=y_1\cdots y_{2m}$.

(ii) Let $P\colon v_1,\ldots,v_{2m+1}$ be a path of even length in $T$, $\mathcal{O}$ be an odd subset of $P$, and assume without loss of generality that $v_i=i$ for $i=1,\ldots,2m+1$. If $\mathcal{O}=\varnothing$, 
\begin{multline*}
g'_{P,\mathcal{O}}=\sum_{i=1}^m\left(-\frac{x_{2i+1}z}{y_{2i-1}y_{2i}y_{2i+1}}(x_{2i-1}x_{2i}+y_{2i-1}y_{2i})+\frac{x_{2i-1}z}{y_{2i-1}y_{2i}y_{2i+1}}(x_{2i}x_{2i+1}+y_{2i}y_{2i+1})\right)
\end{multline*}
belongs to $I$, where $z=y_1\cdots y_{2m+1}$. Now, assume that $\mathcal{O}=\{i_1,\ldots,i_t\}$ is non-empty with $i_1<\cdots<i_t$. Put $\mathcal{O}':=\mathcal{O}\cup\{1,2m+1\}$. Also, let $P_j$ denote the path from $i_j$ to $i_{j+1}$ for all $1\leq j<t$. Then
\[g'_{P,\mathcal{O}}=\sum_{j=1}^m\frac{z}{x_{i_j}x_{i_{j+1}}z'_{P_j,\varnothing}}\cdot g'_{P_j,\varnothing}\]
belongs to $I$, where $z=(\prod_{v\in \mathcal{O}}x_v)\cdot(\prod_{v\in V(P)\setminus \mathcal{O}'}y_v)$.

In what follows, we show that $\mathcal{G}$ is a Gr\"{o}bner basis for $I$ with respect to the lexicographic ordering induced by $x_1>\cdots>x_n>y_1>\cdots>y_n$. Accordingly, we put $\init(w):=\init_\leq(w)$ and $\term(w):=\term_\leq(w)$ for any binomial $w\in S$. Since polynomials with coprime leading terms reduce to zero modulo $\mathcal{G}$ (see \cite[Lemma 2.3.1]{jh-th}), we always assume that the polynomials under considerations have non-coprime leading terms. Also, we consider a superficial direction from $v_1$ to $v_{2m+1}$ on any path $v_1,\ldots,v_{2m+1}$ of type (ii). We have three cases to consider:

\textbf{Case (i,i)}. Let $P\colon u_1,\ldots,u_{2m}$ and $Q:v_1,\ldots,v_{2m'}$ be paths in $T$ of length at least one. Then either $\gcd(\init(g_P),\init(g_Q))|z_P$ or $\gcd(\init(g_P),\init(g_Q))\nmid z_P$. In the former case,
\begin{align*}
S(g_P,g_Q)&=\frac{z_Pz_Qx_{v_1}x_{v_{2m'}}y_{u_1}y_{u_{2m}}}{z}-\frac{z_Pz_Qx_{u_1}x_{u_{2m}}y_{v_1}y_{v_{2m'}}}{z}\\
&=g_Q\parn{\prod_{v\in V(P)\setminus V(Q)}y_v}-g_P\parn{\prod_{v\in V(Q)\setminus V(P)}y_v}
\end{align*}
is in standard form assuming that $x_{v_1}x_{v_{2m'}} > x_{u_1}x_{u_{2m}}$, where $z=\gcd(\init(g_P),\init(g_Q))$. Also, in the latter case,
\[S(g_P,g_Q)=\frac{z_Pz_Qy_{u_1}y_{u_{2m}}x_{v_{j'}}}{z}-\frac{z_Pz_Qx_{u_{i'}}y_{v_1}y_{v_{2m'}}}{z}\]
and either
\[S(g_P,g_Q)=g'_{R,\varnothing}\parn{\prod_{v\in V(P)\cap V(Q)}y_v},\]
is in standard form if $V(P)\subseteq V(Q)$ or $V(Q)\subseteq V(P)$, or 
\[S(g_P,g_Q)=g'_{R,\varnothing}\parn{\prod_{v\in (V(P)\cap V(Q))\setminus\{u_k\}}y_v},\]
is in standard form if neither $V(P)\subseteq V(Q)$ nor $V(Q)\subseteq V(P)$ and $u_k$ is the vertex of degree three in the subgraph induced by $V(P)\cup V(Q)$. Here $R$ is the path between $u_{i'}$ and $v_{j'}$, where $i',j'$ are such that $x_{u_i}z=x_{v_j}z=\gcd(\init(g_P),\init(g_Q))$ if $\{i,i'\}=\{1,2m\}$ and $\{j,j'\}=\{1,2m'\}$.

\textbf{Case (i,ii)}. Let $P\colon u_1,\ldots,u_{2m}$ and $Q:v_1,\ldots,v_{2m'+1}$ be paths in $T$ of length at least one, $\mathcal{O}$ be an odd subset of $Q$, and $\S:=S(g_P,g'_{Q,\mathcal{O}})$. Then either $\gcd(\init(g_P),\init(g'_{Q,\mathcal{O}}))$ is coprime to $x_{u_1}x_{u_{2m}}$ or $\gcd(\init(g_P),\init(g'_{Q,\mathcal{O}}))$ is divisible by $x_{u_1}$.

First assume that $\gcd(\init(g_P),\init(g'_{Q,\mathcal{O}}))$ is coprime to $x_{u_1}x_{u_{2m}}$. We have two cases to consider. If $y_{v_{2m'+1}}\nmid\gcd(\init(g_P),\init(g'_{Q,\mathcal{O}}))$, then 
\begin{align*}
S(g_P,g'_{Q,\mathcal{O}})&=\frac{z_Pz'_{Q,\mathcal{O}}y_{u_1}y_{u_{2m}}x_{v_1}y_{v_{2m'+1}}}{z}+\frac{z_Pz'_{Q,\mathcal{O}}x_{u_1}x_{u_{2m}}y_{v_1}x_{v_{2m'+1}}}{z}\\
&=g'_{Q,\mathcal{O}}\cdot\frac{\init(\S)}{\init(g'_{Q,\mathcal{O}})}+g_P\cdot\frac{\term(\S)}{\init(g_P)},
\end{align*}
is in standard form if $x_{v_1}>x_{u_1}$ and
\begin{align*}
S(g_P,g'_{Q,\mathcal{O}})&=g_P\cdot\frac{\init(\S)}{\init(g_P)}+g'_{Q,\mathcal{O}}\cdot\frac{\term(\S)}{\init(g'_{Q,\mathcal{O}})},
\end{align*}
is in standard form if $x_{u_1}>x_{v_1}$, where $z=\gcd(\init(g_P),\init(g'_{Q,\mathcal{O}}))$. Now assume that $y_{v_{2m'+1}}\mid\gcd(\init(g_P),\init(g'_{Q,\mathcal{O}}))$, $u_i$ is in the connected component of $P-E(Q)$ containing $v_{2m'+1}$, and $i'$ be such that $\{i,i'\}=\{1,2m\}$. Suppose $x_{v_1}>x_{u_1}$. Then
\begin{itemize}
\item if $P_1$ is the path between $u_i$ and $v_{2m'+1}$, $Q_1$ is the path from $v_1$ to $u_{i'}$ of even length, and $\mathcal{O}_1=\mathcal{O}\cap V(Q_1)$, then 
\begin{align*}
S(g_P,g'_{Q,\mathcal{O}})=g'_{Q_1,\mathcal{O}_1}\cdot\frac{\init(\S)}{\init(g'_{Q_1,\mathcal{O}_1})}+g_{P_1}\cdot\frac{\term(\S)}{\init(g_{P_1})}
\end{align*}
is in standard form.
\item if $P_1$ is the path between $u_{i'}$ and $v_{2m'+1}$, $Q_1$ is the path from $v_1$ to $u_i$ of even length, and $\mathcal{O}_1=\mathcal{O}\cap V(Q_1)$, then
\begin{align*}
S(g_P,g'_{Q,\mathcal{O}})=g'_{Q_1,\mathcal{O}_1}\cdot\frac{\init(\S)}{\init(g'_{Q_1,\mathcal{O}_1})}+g_{P_1}\cdot\frac{\term(\S)}{\init(g_{P_1})}
\end{align*}
is in standard form.
\end{itemize}
Next suppose $x_{u_1}>x_{v_1}$. Then
\begin{itemize}
\item if $P_1$ is the path between $u_{i'}$ and $v_{2m'+1}$ of odd length, $Q_1$ is the path from $u_i$ to $v_1$, and $\mathcal{O}_1=\mathcal{O}\cap V(Q_1)$, then
\begin{align*}
S(g_P,g'_{Q,\mathcal{O}})=g_{P_1}\cdot\frac{\init(\S)}{\init(g_{P_1})}-g'_{Q_1,\mathcal{O}_1}\cdot\frac{\term(\S)}{\init(g'_{Q_1,\mathcal{O}_1})}
\end{align*}
is in standard form.
\item if the path $P_1$ between $u_i$ and $v_{2m'+1}$ has odd length, then 
\begin{align*}
S(g_P,g'_{Q,\mathcal{O}})=g_{P_1}\cdot\frac{\init(\S)}{\init(g_{P_1})}+g'_{Q_1,\mathcal{O}_1}\cdot\frac{\term(\S)}{\init(g'_{Q_1,\mathcal{O}_1})}
\end{align*}
is in standard form if $x_{v_1}>x_{u_{2m}}$, $Q_1$ is the path from $v_1$ to $u_{i'}$, and $\mathcal{O}_1=\mathcal{O}\cap V(Q_1)$, and 
\begin{align*}
S(g_P,g'_{Q,\mathcal{O}})=g_{P_1}\cdot\frac{\init(\S)}{\init(g_{P_1})}-g'_{Q_1,\mathcal{O}_1}\cdot\frac{\term(\S)}{\init(g'_{Q_1,\mathcal{O}_1})}
\end{align*}
is in standard form if $x_{u_{2m}}>x_{v_1}$, $Q_1$ is the path from $u_{i'}$ to $v_1$, and $\mathcal{O}_1=\mathcal{O}\cap V(Q_1)$.
\end{itemize}

Finally we consider the latter case that $x_{u_1}\mid \gcd(\init(g_P),\init(g'_{Q,\mathcal{O}}))$. Observe that $x_{v_{2m'+1}}\nmid \gcd(\init(g_P),\init(g'_{Q,\mathcal{O}}))$. First assume that $x_{v_1}\nmid\gcd(\init(g_P),\init(g'_{Q,\mathcal{O}}))$. Then for a suitable choices of $i,i'$ with $\{i,i'\}=\{1,2m\}$, one of the following two cases occurs:
\begin{itemize}
\item if $x_{v_1}>x_{u_1}$, $Q_1$ is the path from $v_1$ to $v_k$, $Q_2$ is the path from $v_k$ to $v_{2m'+1}$, and $\mathcal{O}_t=\mathcal{O}\cap V(Q_t)$ for $t=1,2$, then
\begin{align*}
S(g_P,g'_{Q,\mathcal{O}})=&\frac{z_Pz'_Qy_{u_1}y_{u_{2m}}x_{v_1}y_{v_{2m'+1}}}{zx_{u_i}}+\frac{z_Pz'_{Q,\mathcal{O}}x_{u_{i'}}y_{v_1}x_{v_{2m'+1}}}{z}\\
=&g'_{Q_1,\mathcal{O}_1}\cdot\frac{\init(\S)}{\init(g'_{Q_1,\mathcal{O}_1})}+g_P\cdot\frac{\init(\S)}{\init(g_P)}+g'_{Q_2,\mathcal{O}_2}\cdot\frac{\init(\S)y_{v_1}x_{v_k}}{\init(g'_{Q_2,\mathcal{O}_2})x_{v_1}y_{v_k}}
\end{align*}
is in standard form.
\item if $x_{u_1}>x_{v_1}$ with $Q_1$ and $Q_2$ as above, where $u_i=v_k$ for some $k\in\mathcal{O}$, $x_{u_i}z=\gcd(\init(g_P),\init(g'_{Q,\mathcal{O}}))$, and $x_{v_k}>x_{2m'+1}$ w.l.o.g., then
\begin{align*}
S(g_P,g'_{Q,\mathcal{O}})=g_P\cdot\frac{\init(\S)}{\init(g_P)}+g'_{Q_1,\mathcal{O}_1}\cdot\frac{\term(\S)}{\init(g'_{Q_1,\mathcal{O}_1})}+g'_{Q_2,\mathcal{O}_2}\cdot\frac{\term(\S)y_{v_1}x_{v_k}}{\init(g'_{Q_2,\mathcal{O}_2})x_{v_1}y_{v_k}}
\end{align*}
is in standard form.
\end{itemize}

Next, if $x_{v_1}\mid \gcd(\init(g_P),\init(g'_{Q,\mathcal{O}}))$, then $x_{v_1}=x_{u_i}$ and 
\begin{align*}
S(g_P,g'_{Q,\mathcal{O}})=g'_{Q,\mathcal{O}}\cdot\frac{\init(\S)}{\init(g'_{Q,\mathcal{O}})}+g_P\cdot\frac{\term(\S)}{\init(g_P)}
\end{align*}
is in standard form.

\textbf{Case (ii,ii)}. Let $P\colon u_1,\ldots,u_{2m+1}$ and $Q\colon v_1,\ldots,v_{2m'+1}$ be paths in $T$ of even lengths, $\mathcal{O}_P$ and $\mathcal{O}_Q$ be odd subsets of $P$ and $Q$, respectively, $x_{u_1}\geq x_{v_1}$ by symmetry, and $\S:=S(g'_{P,\mathcal{O}_P},g'_{Q,\mathcal{O}_Q})$. First observe that $\S=0$ when $P=Q$. Hence, in what follows, we assume that $P\neq Q$. Let $z^*:=\gcd(\init(g'_{P,\mathcal{O}_P}),\init(g'_{Q,\mathcal{O}_Q}))$. We have the following four cases:

Case 1. $z$ is coprime to $y_{u_{2m+1}}y_{v_{2m'+1}}$. Then
\[S(g'_{P,\mathcal{O}_P},g'_{Q,\mathcal{O}_Q})=\frac{z'_{P,\mathcal{O}_P}z'_{Q,\mathcal{O}_Q}x_{u_1}y_{v_{2m+1}}y_{v_1}x_{v_{2m'+1}}}{z^*}-\frac{z'_{P,\mathcal{O}_P}z'_{Q,\mathcal{O}_Q}y_{u_1}x_{u_{2m+1}}x_{v_1}y_{v_{2m'+1}}}{z^*}\]
and one of the followings occurs:
\begin{itemize}
\item Subcase 1.1. if $v_1\neq u_1\notin\mathcal{O}_Q$, then 
\begin{align*}
S(g'_{P,\mathcal{O}_P},g'_{Q,\mathcal{O}_Q})=g'_{P,\mathcal{O}_P}\cdot\frac{\init(\S)}{\init(g'_{P,\mathcal{O}_P})}-g'_{Q,\mathcal{O}_Q}\cdot\frac{\term(\S)}{\init(g'_{Q,\mathcal{O}_Q})},
\end{align*}
is in standard form.
\item Subcase 1.2. if $v_1=u_1\notin\mathcal{O}_Q$, $R$ is the path from $u_{2m+1}$ to $v_{2m'+1}$ (assuming that $x_{u_{2m+1}}>x_{v_{2m'+1}}$ by symmetry), and $\mathcal{O}_R=(\mathcal{O}_P\cap V(R))\cup(\mathcal{O}_Q\cap V(R))$, then
\begin{align*}
S(g'_{P,\mathcal{O}_P},g'_{Q,\mathcal{O}_Q})=g'_{R,\mathcal{O}_R}\cdot\frac{\init(\S)}{\init(g'_{R,\mathcal{O}_R})},
\end{align*}
is in standard form.
\item Subcase 1.3. if $u_1\in\mathcal{O}_Q$, then 
\begin{align*}
S(g'_{P,\mathcal{O}_P},g'_{Q,\mathcal{O}_Q})=g'_{P,\mathcal{O}_P}\cdot\frac{\init(\S)}{\init(g'_{P,\mathcal{O}_P})}-g'_{Q,\mathcal{O}_Q\setminus\{u_1\}}\cdot\frac{\term(\S)}{\init(g'_{Q,\mathcal{O}_Q\setminus\{u_1\}})},
\end{align*}
is in standard form.
\end{itemize}

Case 2. $z^*=zy_{u_{2m+1}}$ is coprime to $y_{v_{2m'+1}}$. Then
\[S(g'_{P,\mathcal{O}_P},g'_{Q,\mathcal{O}_Q})=\frac{z'_{P,\mathcal{O}_P}z'_{Q,\mathcal{O}_Q}x_{u_1}y_{v_1}x_{v_{2m'+1}}}{z}-\frac{z'_{P,\mathcal{O}_P}z'_{Q,\mathcal{O}_Q}y_{u_1}x_{u_{2m+1}}x_{v_1}y_{v_{2m'+1}}}{zy_{u_{2m+1}}}\]
and one of the followings occurs:
\begin{itemize}
\item Subcase 2.1. $u_1=v_1$. Let $Q_1$ be the path between $u_{2m+1}$ and $v_{2m'+1}$ and $\mathcal{O}_{Q_1}=\mathcal{O}_Q\cap V(Q_1)$. Then 
\[S(g'_{P,\mathcal{O}_P},g'_{Q,\mathcal{O}_Q})=\pm g'_{Q_1,\mathcal{O}_{Q_1}}\cdot\frac{\init(\S)}{\init(g'_{Q_1,\mathcal{O}_{Q_1}})}\]
is in standard form with sign being $+$ or $-$ according to $x_{v_{2m'+1}}>x_{u_{2m+1}}$ or not.
\item Subcase 2.2. $u_1\neq v_1$, $u_1\notin\mathcal{O}_Q$, and $v_1\notin\mathcal{O}_P$. If $\mathcal{O}_Q\cup\{u_{2m+1}\}$ is an odd subset of $Q$, then
\[S(g'_{P,\mathcal{O}_P},g'_{Q,\mathcal{O}_Q})=g'_{P,\mathcal{O}_P}\cdot\frac{\init(\S)}{\init(g'_{P,\mathcal{O}_P})}-g'_{Q,\mathcal{O}_Q}\cdot\frac{\term(\S)}{\init(g'_{Q,\mathcal{O}_Q})}\]
is in standard form. Now assume that $\mathcal{O}_Q\cup\{u_{2m+1}\}$ is not an odd subset of $Q$. Let $\mathcal{O}_Q=\{v_{i_1},\ldots,v_{i_m}\}$ with $i_1<\cdots<i_m$, and $s$ be such that $u_{2m+1}$ lies on the path between $v_{i_s}$ and $v_{i_{s+1}}$. Let $P_1$ be the path between $v_{i_s}$ and $u_{2m+1}$, $P_2$ be the path between $u_{2m+1}$ and $v_{i_{s+1}}$, $Q_1$ be the path between $v_1$ and $v_{i_s}$, and $Q_2$ be the path between $v_{i_{s+1}}$ and $v_{2m'+1}$. Notice that $P_1$ and $P_2$ are paths of odd lengths and $Q_1$ and $Q_2$ are paths of even lengths, and that $\mathcal{O}_{Q_i}:=\mathcal{O}_Q\cap V(Q_i)$ is an odd subset of $V(Q_i)$ for $i=1,2$. In the case where $Q_i$ ($i=1,2$) is a path of length zero, we define $g'_{Q_i,\mathcal{O}_{Q_i}}:=1$. Also, put $\varepsilon_i:=1-\delta_{|V(Q_i)|,1}$ for $i=1,2$. Then
\begin{align*}
S(g'_{P,\mathcal{O}_P},g'_{Q,\mathcal{O}_Q})=&g'_{P,\mathcal{O}_P}\cdot\frac{\init(\S)}{\init(g'_{P,\mathcal{O}_P})}
-g_{P_1}\cdot\frac{\term(\S)}{\init(g_{P_1})}
+\varepsilon_1 g'_{Q_1,\mathcal{O}_{Q_1}}\cdot\frac{\term(\S)y_{v_{i_s}}y_{u_{2m+1}}}{\init(g'_{Q_1,\mathcal{O}_{Q_1}})x_{v_{i_s}}x_{u_{2m+1}}}\\
&+g_{P_2}\cdot\frac{\init(\S)y_{u_1}x_{u_{2m+1}}}{\init(g_{P_2})x_{u_1}y_{u_{2m+1}}}
\pm \varepsilon_2 g'_{Q_2,\mathcal{O}_{Q_2}}\cdot\frac{\term(\S)y_{v_{i_s}}y_{u_{2m+1}}y_{v_1}x_{v_{i_s}}}{\init(g'_{Q_2,\mathcal{O}_{Q_2}})x_{v_{i_s}}x_{u_{2m+1}}x_{v_1}y_{v_{i_s}}}
\end{align*}
is in standard form with sign being $+$ or $-$ according to $Q_1$ being in direction of $Q$ or not.
\item Subcase 2.3. $u_1\neq v_1$, $u_1\in\mathcal{O}_Q$, and $v_1\notin\mathcal{O}_P$. Let $\mathcal{O}'_Q:=(\mathcal{O}_Q\cup\{u_{2m+1}\})\setminus\{u_1\}$. Then 
\[S(g'_{P,\mathcal{O}_P},g'_{Q,\mathcal{O}_Q})=g'_{P,\mathcal{O}_P}\cdot\frac{\init(\S)}{\init(g'_{P,\mathcal{O}_P})}-g'_{Q,\mathcal{O}'_Q}\cdot\frac{\term(\S)}{\init(g'_{Q,\mathcal{O}'_Q})}\]
is in standard form.
\item Subcase 2.4. $u_1\neq v_1$, $u_1\notin\mathcal{O}_Q$, and $v_1\in\mathcal{O}_P$. Let $\mathcal{O}'_P:=\mathcal{O}_P\setminus\{v_1\}$ and $\mathcal{O}'_Q:=\mathcal{O}_Q\cup\{u_{2m+1}\}$. Then 
\[S(g'_{P,\mathcal{O}_P},g'_{Q,\mathcal{O}_Q})=g'_{P,\mathcal{O}'_P}\cdot\frac{\init(\S)}{\init(g'_{P,\mathcal{O}'_P})}-g'_{Q,\mathcal{O}'_Q}\cdot\frac{\term(\S)}{\init(g'_{Q,\mathcal{O}'_Q})}\]
is in standard form.
\end{itemize}
Case 3. $z^*=zy_{v_{2m'+1}}$ is coprime to $y_{u_{2m+1}}$. Then
\[S(g'_{P,\mathcal{O}_P},g'_{Q,\mathcal{O}_Q})=\frac{z'_{P,\mathcal{O}_P}z'_{Q,\mathcal{O}_Q}x_{u_1}y_{u_{2m+1}}y_{v_1}x_{v_{2m'+1}}}{zy_{v_{2m'+1}}}-\frac{z'_{P,\mathcal{O}_P}z'_{Q,\mathcal{O}_Q}y_{u_1}x_{u_{2m+1}}x_{v_1}}{z}\]
and one of the followings occurs:
\begin{itemize}
\item Subcase 3.1. $u_1=v_1$. Let $P_1$ be the path between $u_{2m+1}$ and $v_{2m'+1}$, and $\mathcal{O}_{P_1}=\mathcal{O}_Q\cap V(P_1)$. Then 
\[S(g'_{P,\mathcal{O}_P},g'_{Q,\mathcal{O}_Q})=\pm g'_{P_1,\mathcal{O}_{P_1}}\cdot\frac{\init(\S)}{\init(g'_{P_1,\mathcal{O}_{P_1}})}\]
is in standard form with sign being $+$ or $-$ according to $x_{v_{2m'+1}}>x_{u_{2m+1}}$ or not.
\item Subcase 3.2. $u_1\neq v_1$, $u_1\notin\mathcal{O}_Q$, and $v_1\notin\mathcal{O}_P$. If $\mathcal{O}'_P:=\mathcal{O}_P\cup\{v_{2m'+1}\}$ is an odd subset of $P$, then
\[S(g'_{P,\mathcal{O}_P},g'_{Q,\mathcal{O}_Q})=g'_{P,\mathcal{O}'_P}\cdot\frac{\init(\S)}{\init(g'_{P,\mathcal{O}'_P})}-g'_{Q,\mathcal{O}_Q}\cdot\frac{\term(\S)}{\init(g'_{Q,\mathcal{O}_Q})}\]
is in standard form. Now assume that $\mathcal{O}_P\cup\{v_{2m'+1}\}$ is not an odd subset of $P$. Let $\mathcal{O}_P=\{u_{i_1},\ldots,u_{i_m}\}$ with $i_1<\cdots<i_m$, and $s$ be such that $v_{2m'+1}$ lies on the path between $u_{i_s}$ and $u_{i_{s+1}}$. Let $P_1$ be the path between $u_{i_s}$ and $v_{2m'+1}$, $P_2$ be the path between $v_{2m'+1}$ and $u_{i_{s+1}}$, and $P_3$ be the path between the last member $v_{\max}$ of $\{v_1\}\cup\mathcal{O}_Q$ and $u_{i_{s+1}}$. Notice that $P_1$, $P_2$, and $P_3$ are paths of odd lengths. Then
\begin{align*}
S(g'_{P,\mathcal{O}_P},g'_{Q,\mathcal{O}_Q})=&g_{P_2}\cdot\frac{\init(\S)}{\init(g_{P_2})}-g'_{P,\mathcal{O}_P\setminus\{u_{i_{s+1}}\}}\cdot\frac{\init(\S)y_{u_{i_{s+1}}}}{\init(g'_{P,\mathcal{O}_P\setminus\{u_{i_{s+1}}\}})x_{u_{i_{s+1}}}}\\
&-g'_{Q,\mathcal{O}_Q}\cdot\frac{\term(\S)}{\init(g'_{Q,\mathcal{O}_Q})}-g_{P_2}\cdot\frac{\term(\S)y_{v_1}x_{v_{2m'+1}}}{\init(g_{P_2})x_{v_1}y_{v_{2m'+1}}}
\end{align*}
is in standard form if $x_{u_1}>x_{u_{i_{s+1}}}$ and $x_{v_1}>x_{u_{i_{s+1}}}$, 
\begin{align*}
S(g'_{P,\mathcal{O}_P},g'_{Q,\mathcal{O}_Q})=&g_{P_2}\cdot\frac{\init(\S)}{\init(g_{P_2})}-g'_{P,\mathcal{O}_P\setminus\{u_{i_{s+1}}\}}\cdot\frac{\init(\S)y_{u_{i_{s+1}}}}{\init(g'_{P,\mathcal{O}_P\setminus\{u_{i_{s+1}}\}})x_{u_{i_{s+1}}}}\\
&-g_{P_3}\cdot\frac{\term(\S)}{\init(P_3)}+\varepsilon g'_{Q,\mathcal{O}_Q\setminus\{v_{\max}\}}\cdot\frac{\term(\S)y_{v_{\max}}y_{v_{i_{s+1}}}}{\init(g'_{Q,\mathcal{O}_Q\setminus\{v_{\max}\}})x_{v_{\max}}x_{v_{i_{s+1}}}}
\end{align*}
is in standard form if $x_{u_1}>x_{u_{i_{s+1}}}$ and $x_{v_1}<x_{u_{i_{s+1}}}$ with $\varepsilon=1-\delta_{\mathcal{O}_Q,\varnothing}$, 
\begin{align*}
S(g'_{P,\mathcal{O}_P},g'_{Q,\mathcal{O}_Q})=&g_{P_2}\cdot\frac{\init(\S)}{\init(g_{P_2})}-g'_{Q,\mathcal{O}_Q}\cdot\frac{\term(\S)}{\init(g'_{Q,\mathcal{O}_Q})}\\
&-g_{P_2}\cdot\frac{\term(\S)y_{v_1}x_{v_{2m'+1}}}{\init(g_{P_2})x_{v_1}y_{v_{2m'+1}}}
-g'_{P,\mathcal{O}_P\setminus\{u_{i_{s+1}}\}}\cdot\frac{\init(\S)y_{u_{i_{s+1}}}}{\init(g'_{P,\mathcal{O}_P\setminus\{u_{i_{s+1}}\}})x_{u_{i_{s+1}}}}
\end{align*}
is in standard form if $x_{u_1}<x_{u_{i_{s+1}}}$ and $x_{v_1}>x_{u_{i_{s+1}}}$, 
\begin{align*}
S(g'_{P,\mathcal{O}_P},g'_{Q,\mathcal{O}_Q})=&g_{P_2}\cdot\frac{\init(\S)}{\init(g_{P_2})}-g_{P_3}\cdot\frac{\term(\S)}{\init(P_3)}-g'_{P,\mathcal{O}_P\setminus\{u_{i_{s+1}}\}}\cdot\frac{\init(\S)y_{u_{i_{s+1}}}}{\init(g'_{P,\mathcal{O}_P\setminus\{u_{i_{s+1}}\}})x_{u_{i_{s+1}}}}\\
&+\varepsilon g'_{Q,\mathcal{O}_Q\setminus\{v_{\max}\}}\cdot\frac{\term(\S)y_{v_{\max}}y_{v_{i_{s+1}}}}{\init(g'_{Q,\mathcal{O}_Q\setminus\{v_{\max}\}})x_{v_{\max}}x_{v_{i_{s+1}}}}
\end{align*}
is in standard form if $x_{u_1}<x_{u_{i_{s+1}}}$ and $x_{v_1}<x_{u_{i_{s+1}}}$ with $\varepsilon=1-\delta_{\mathcal{O}_Q,\varnothing}$.
\item Subcase 3.3. $u_1\neq v_1$, $u_1\in\mathcal{O}_Q$, and $v_1\notin\mathcal{O}_P$. Let $\mathcal{O}'_P:=\mathcal{O}_P\cup\{v_{2m'+1}\}$ and $\mathcal{O}'_Q:=\mathcal{O}_Q\setminus\{u_1\}$. Then 
\[S(g'_{P,\mathcal{O}_P},g'_{Q,\mathcal{O}_Q})=g'_{P,\mathcal{O}'_P}\cdot\frac{\init(\S)}{\init(g'_{P,\mathcal{O}'_P})}-g'_{Q,\mathcal{O}'_Q}\cdot\frac{\term(\S)}{\init(g'_{Q,\mathcal{O}'_Q})}\]
is in standard form.
\item Subcase 3.4. $u_1\neq v_1$, $u_1\notin\mathcal{O}_Q$, and $v_1\in\mathcal{O}_P$. Let $\mathcal{O}'_P:=(\mathcal{O}_P\cup\{v_{2m'+1}\})\setminus\{v_1\}$. Then 
\[S(g'_{P,\mathcal{O}_P},g'_{Q,\mathcal{O}_Q})=g'_{P,\mathcal{O}'_P}\cdot\frac{\init(\S)}{\init(g'_{P,\mathcal{O}'_P})}-g'_{Q,\mathcal{O}_Q}\cdot\frac{\term(\S)}{\init(g'_{Q,\mathcal{O}_Q})}\]
is in standard form.
\end{itemize}
Case 4. $z^*$ is divisible by $y_{u_{2m+1}}$ and $y_{v_{2m'+1}}$. If $u_{2m+1}=v_{2m'+1}$, then 
\begin{align*}
S(g'_{P,\mathcal{O}_P},g'_{Q,\mathcal{O}_Q})&=\frac{z'_{P,\mathcal{O}_P}z'_{Q,\mathcal{O}_Q}x_{u_1}x_{u_{2m+1}}y_{v_1}}{z}-\frac{z'_{P,\mathcal{O}_P}z'_{Q,\mathcal{O}_Q}y_{u_1}x_{u_{2m+1}}x_{v_1}}{z}\\
&=g'_{R,\mathcal{O}_R}\cdot\frac{\init(\S)}{\init(g'_{R,\mathcal{O}_R})},
\end{align*}
where $R$ is the path from $u_1$ to $v_1$ and $\mathcal{O}_R=V(R)\cap(\mathcal{O}_P\cup\mathcal{O}_Q)$. Next assume that $u_{2m+1}\neq v_{2m'+1}$. Then 
\[S(g'_{P,\mathcal{O}_P},g'_{Q,\mathcal{O}_Q})=\frac{z'_{P,\mathcal{O}_P}z'_{Q,\mathcal{O}_Q}x_{u_1}y_{v_1}x_{v_{2m'+1}}}{zy_{v_{2m'+1}}}-\frac{z'_{P,\mathcal{O}_P}z'_{Q,\mathcal{O}_Q}y_{u_1}x_{u_{2m+1}}x_{v_1}}{zy_{u_{2m+1}}}.\]
Clearly, 
\[S(g'_{P,\mathcal{O}_P},g'_{Q,\mathcal{O}_Q})=g'_{P,\mathcal{O}_P\cup\{v_{2m'+1}\}}\cdot\frac{\init(\S)}{\init(g'_{P,\mathcal{O}_P\cup\{v_{2m'+1}\}})}-g'_{Q,\mathcal{O}_Q\cup\{u_{2m+1}\}}\cdot\frac{\term(\S)}{\init(g'_{Q,\mathcal{O}_Q\cup\{u_{2m+1}\}})}\]
is in standard form if $\mathcal{O}_P\cup\{v_{2m'+1}\}$ is an odd subset of $P$. Now assume that $\mathcal{O}_P\cup\{v_{2m'+1}\}$ is not an odd subset of $P$. Let $u_{\max}$ (resp. $v_{\max}$) be the last element of $\{u_1\}\cup\mathcal{O}_P$ (resp. $\{v_1\}\cup\mathcal{O}_Q$) that belongs to the path between $u_1$ and $v_{2m'+1}$ (resp. $v_1$ and $u_{2m+1}$). Also, let $P_1$ be the path between $u_{\max}$ and $v_{2m'+1}$, $P_2$ be the path between $u_1$ and $u_{\max}$, $Q_1$ be the path between $v_{\max}$ and $u_{2m+1}$, and $Q_2$ be the path between $v_1$ and $v_{\max}$. We have four cases to consider:
\begin{itemize}
\item Subcase 4.1. $u_1=u_{\max}$ and $v_1=v_{\max}$. Then
\[S(g'_{P,\mathcal{O}_P},g'_{Q,\mathcal{O}_Q})=g_{P_1}\cdot\frac{\init(\S)}{\init(g_{P_1})}-g_{Q_1}\cdot\frac{\term(\S)}{\init(g_{Q_1})}\]
is in standard form.
\item Subcase 4.2. $u_1=u_{\max}$ and $v_1\neq v_{\max}$. Then
\[S(g'_{P,\mathcal{O}_P},g'_{Q,\mathcal{O}_Q})=g_{P_1}\cdot\frac{\init(\S)}{\init(g_{P_1})}-g_{Q_1}\cdot\frac{\term(\S)}{\init(g_{Q_1})}\pm g'_{Q_2,\mathcal{O}_Q\cap V(Q_2)}\cdot\frac{\term(\S)y_{u_{2m+1}}y_{v_{\max}}}{\init(g'_{Q_2,\mathcal{O}_Q\cap V(Q_2)})x_{u_{2m+1}}x_{v_{\max}}}\]
is in standard form with sign being $+$ or $-$ according to $Q_2$ be in direction of $Q$ or not.
\item Subcase 4.3. $u_1\neq u_{\max}$ and $v_1=v_{\max}$. Then
\[S(g'_{P,\mathcal{O}_P},g'_{Q,\mathcal{O}_Q})=g_{P_1}\cdot\frac{\init(\S)}{\init(g_{P_1})}-g'_{P_2,\mathcal{O}_P\cap V(P_2)}\cdot\frac{\init(\S)y_{u_{\max}}y_{v_{2m'+1}}}{\init(g_{P_2,\mathcal{O}_P\cap V(P_2)})x_{u_{\max}}x_{v_{2m'+1}}}-g_{Q_1}\cdot\frac{\term(\S)}{\init(g_{Q_1})}\]
is in standard form if $x_{u_1}>x_{u_{\max}}$ and
\[S(g'_{P,\mathcal{O}_P},g'_{Q,\mathcal{O}_Q})=g_{P_1}\cdot\frac{\init(\S)}{\init(g_{P_1})}-g'_{P_2,\mathcal{O}_P\cap V(P_2)}\cdot\frac{\term(\S)}{\init(g_{P_2,\mathcal{O}_P\cap V(P_2)})}-g_{Q_1}\cdot\frac{\term(\S)x_{u_1}y_{u_{\max}}}{\init(g_{Q_1})y_{u_1}x_{u_{\max}}}\]
is in standard form if $x_{u_1}<x_{u_{\max}}$.
\item Subcase 4.4. $u_1\neq u_{\max}$ and $v_1\neq v_{\max}$. Then
\begin{align*}
S(g'_{P,\mathcal{O}_P},g'_{Q,\mathcal{O}_Q})=&g_{P_1}\cdot\frac{\init(\S)}{\init(g_{P_1})}-g'_{P_2,\mathcal{O}_P\cap V(P_2)}\cdot\frac{\init(\S)y_{u_{\max}}y_{v_{2m'+1}}}{\init(g_{P_2,\mathcal{O}_P\cap V(P_2)})x_{u_{\max}}x_{v_{2m'+1}}}\\
&-g_{Q_1}\cdot\frac{\term(\S)}{\init(g_{Q_1})}\pm g'_{Q_2,\mathcal{O}_Q\cap V(Q_2)}\cdot\frac{\term(\S)y_{u_{2m+1}}y_{v_{\max}}}{\init(g'_{Q_2,\mathcal{O}_Q\cap V(Q_2)})x_{u_{2m+1}}x_{v_{\max}}}
\end{align*}
is in standard form if $x_{u_1}>x_{u_{\max}}$ and
\begin{align*}
S(g'_{P,\mathcal{O}_P},g'_{Q,\mathcal{O}_Q})=&g_{P_1}\cdot\frac{\init(\S)}{\init(g_{P_1})}-g_{P_2,\mathcal{O}_P\cap V(P_2)}\cdot\frac{\term(\S)}{\init(g_{P_2,\mathcal{O}_P\cap V(P_2)})}-g_{Q_1}\cdot\frac{\term(\S)x_{u_1}y_{u_{\max}}}{\init(g_{Q_1})y_{u_1}x_{u_{\max}}}\\
&\pm g'_{Q_2,\mathcal{O}_Q\cap V(Q_2)}\cdot\frac{\term(\S)x_{u_1}y_{u_{\max}}y_{u_{2m+1}}y_{v_{\max}}}{\init(g'_{Q_2,\mathcal{O}_Q\cap V(Q_2)})y_{u_1}x_{u_{\max}}x_{u_{2m+1}}x_{v_{\max}}}
\end{align*}
is in standard form if $x_{u_1}<x_{u_{\max}}$ with sign being $+$ or $-$ according to $Q_2$ be in direction of $Q$ or not.
\end{itemize}
Finally, as an ideal with square-free initial ideal, $I$ is actually a radical ideal (see also \cite[Theorems 1.1--1.2]{jh-am-ssm-vw}). The proof is complete.
\end{proof}

Let $T$ be a tree with $n$ vertices. A labeling $\ell\colon V(T)\rightarrow [n]$ is called \textit{ascending} if it has a pendant receiving the label $n$ whose removal yields a tree with ascending labeling. One observes that every tree has an ascending labeling. It is obvious from the definition that every vertex of maximum label in a subgraph of a tree $T$ with ascending labeling is either an isolated vertex or a pendant. As a result, every path $P\colon v_1,\ldots,v_m$ of $T$ is unimodal in the sense that
\[\ell(v_1)>\cdots>\ell(v_{k-1})>\ell(v_k)<\ell(v_{k+1})<\cdots<\ell(v_m)\]
for some $k$. In what follows, we rename every vertex $v$ of $T$ with its label $\ell(v)$ and assume without loss of generality that $v=\ell(v)$ for all $v\in V(T)$.
\begin{corollary}\label{Groebner basis of LSS_T(2): T has descending labeling}
Let $T$ be a tree with an ascending labeling. Then 
\[\mathcal{G}:=\{g_P \colon \quad P\emph{ is a path of odd length in }T\}\cup\{g'_{P,\varnothing} \colon \quad P\emph{ is a path of even length in }T\}\]
is a Gr\"{o}bner basis for $L_T^{\KK}(2)$ with respect to the lexicographic ordering induced by $x_1>\cdots>x_n>y_1>\cdots>y_n$.
\end{corollary}
\begin{proof}
First observe that for any path in $T$ one of the end points has the maximum label among the other vertices of the path. From the proof of Theorem \ref{Groebner basis of LSS_T(2)} one observes that $S(g_P,g_Q)$ and $S(g_P,g'_{Q',\varnothing})$ reduce to zero modulo $\mathcal{G}$ for all paths $P$ and $Q$  not both of type (ii). Now, let $P\colon u_1,\ldots,u_{2m+1}$ and $Q\colon v_1,\ldots,v_{2m'+1}$ be paths of type (ii) with $\mathcal{O}_P=\mathcal{O}_Q=\varnothing$. Following the proof of Theorem \ref{Groebner basis of LSS_T(2)}, $S(g'_{P,\varnothing},g'_{Q,\varnothing})$ reduces to zero modulo $\mathcal{G}$ except possibly for the following two cases, which we consider independently:

(a) $\gcd(\init(g'_{P,\varnothing}),\init(g'_{Q,\varnothing}))=zy_{v_{2m'+1}}$ is coprime to $y_{u_{2m+1}}$ and $u_1\neq v_1$. As $v_{2m'+1}$ lies on $P$, we have $v_{2m'+1}<u_{2m+1}$. Let $P_1$ be the path from $v_{2m'+1}$ to $u_{2m+1}$ and $R_1$ be the path from $u_1$ to $v_1$. Also, let $P_2$ be the path between $u_1$ and $v_{2m'+1}$ and $R_2$ be the path between $v_1$ and $u_{2m+1}$. Then either $|V(P_1)|$ is odd and 
\[S(g'_{P,\varnothing},g'_{Q,\varnothing})=g'_{P_1,\varnothing}\cdot\frac{\init(\S)}{\init(g'_{P_1,\varnothing})}+g'_{R_1,\varnothing}\cdot\frac{\init(\S)x_{u_{2m+1}}y_{v_{2m'+1}}}{\init(g'_{P_1,\varnothing})y_{u_{2m+1}}x_{v_{2m'+1}}}\]
is in standard form, or $|V(P_1)|$ is even and 
\[S(g'_{P,\varnothing},g'_{Q,\varnothing})=g_{P_2}\cdot\frac{\init(\S)}{\init(g_{P_2})}+g_{R_2}\cdot\frac{\term(\S)}{\init(g_{R_2})}\]
is in standard form.

(b) $\gcd(\init(g'_{P,\varnothing}),\init(g'_{Q,\varnothing}))$ is divisible by both $y_{u_{2m+1}}$ and $y_{v_{2m'+1}}$, where $u_{2m+1}\neq v_{2m'+1}$. Since $v_{2m'+1}$ lies on $P$ and $u_{2m+1}$ lies on $Q$, we obtain $v_{2m'+1}<u_{2m+1}$ and $u_{2m+1}<v_{2m'+1}$, which is impossible.
\end{proof}
\begin{example}
Utilizing Corollary \ref{Groebner basis of LSS_T(2): T has descending labeling}, one observes that
\begin{itemize}
\item[(i)]if $S_n$ be the star graph of order $n$ with central vertex labeled with $2$, then
\[\mathcal{G}:=\{x_1x_i+y_1y_i\colon 2\leq i\leq n\}\cup\{x_iy_1y_j-x_jy_1y_i:2\leq i<j \leq n\}\]
is a reduced Gr\"{o}bner basis for $L_{S_n}^{\KK}(2)$, and
\item[(ii)]if $P_n:1,\ldots,n$ be the path of order $n$, then 
\[\mathcal{G}:=\{x_ix_{i+1}+y_iy_{i+1}:1\leq i\leq n-1\}\cup\{x_iy_{i+1}y_{i+2}-y_iy_{i+1}x_{i+2}:1\leq i \leq n-2\}\]
is a reduced Gr\"{o}bner basis for $L_{P_n}^{\KK}(2)$.
\end{itemize}
\end{example}
\section{Krull dimension of LSS-ideals of trees}
The aim of this section is to present a formula for the Hilbert series of LSS-ideals of trees and the Krull dimension of the quotient ring. To this end, we use the fact that $S/I$ and $S/\init_\leq(I)$ share the same Hilbert series for all monomial orderings $\leq$ (see \cite[Proposition 2.2.5]{jh-th}). In this regard, Corollary \ref{Groebner basis of LSS_T(2): T has descending labeling} provides us with a simple description of $\init_\leq(I)$. Since $\init_\leq(L_T^\KK(2))$ is a square-free monomial ideal, it is the Stanley-Reisner ideal of a simplicial complex.

A collection $\Delta$ of subsets of a set $V=\{v_{1}, \ldots, v_{n} \}$ is a \textit{simplicial complex} if
\begin{itemize}
\item[(a)] $\{ v_{i} \} \in \Delta $ for all $i$;
\item[(b)] if $F\in \Delta$, then all subsets of $F$ are also in $\Delta$ (including the empty set).
\end{itemize}
The elements of $\Delta$ are called \textit{faces} and the maximal faces under inclusion are called \textit{facets} of $\Delta$. A subset $F\subseteq [n]$ is called a \textit{non-face} of $\Delta$ if $F\notin \Delta$.

The \textit{dimension} of a face $F$ is $\dim F = |F|-1$, where $|F|$ denotes the cardinality of $F$ and the \textit{dimension} of $\Delta$, $\dim(\Delta)$, is defined as
\[\dim (\Delta) = \max\{\dim F \colon F \in \Delta \}.\]
Let $\Delta$ be a simplicial complex of dimension $d-1$ on the vertex set $V$ and  $f_i = f_i(\Delta) $ denote the number of faces of $\Delta$ of dimension $i$. Thus, in particular $f_{-1} = 1$ and $f_0 = n$. The sequence
\begin{equation*}
\textbf{f}(\Delta) := \left( f_0, \ldots, f_{d-1} \right)
\end{equation*}
is called the $\textbf{f}$-\textit{vector} of $\Delta$.

To each simplicial complex $\Delta$ on the vertex set $V$ we associate a square-free monomial ideal $I_\Delta\subset\KK[x_1,\ldots,x_n]$, called the Stanley-Reisner ideal of $\Delta$, defined as follows:
\[I_\Delta=(\prod_{i \in F} x_i \colon \, F\notin \Delta ).\]

Now, let $R:=\KK[x_1,\ldots,x_n]$ and $J\subseteq (x_1,\ldots,x_n)^2$ be a square-free monomial ideal. Let $\Supp(u):=\{i\in[n]\colon x_i\mid u\}$ for any monomial $u\in R$ and let $\Delta(J)$ be the set of all subsets $A$ of $[n]$ for which $\Supp(u)\nsubseteq A$ for all $u\in\mathcal{G}(J)$. Then $\Delta(J)$ is a simplicial complex and $I_{\Delta(J)}=J$, hence
\begin{equation}\label{Hilb(R/I,t)}
\Hilb\left(\frac{R}{J},t\right)=\sum_{i=0}^d f_{i-1}\left(\frac{t}{1-t}\right)^i,
\end{equation}
where $d=\dim\Delta(J)+1$ and $f_{i-1}=f_{i-1}(\Delta(J))$ for $i=0,\ldots,d$ (see \cite[Theorem 5.1.7]{wb-jh}). In order to compute $\Hilb(R/J,t)$, we note from the exact sequence 
\[0\longrightarrow\frac{R}{J\colon u}(-d)\overset{\cdot u}{\longrightarrow}\frac{R}{J}\longrightarrow\frac{R}{J+(u)}\longrightarrow0\]
that
\begin{equation}\label{Hilb(S/I)=t^dHilb(S/I:u)+Hilb(S/I+(u))}
\Hilb\left(\frac{R}{J}, t\right)=t^d\Hilb\left(\frac{R}{J\colon u}, t\right)+\Hilb\left(\frac{R}{J+(u)}, t\right).
\end{equation}
For any monomial $u\in R$, define
\[J\odot_i u:=\begin{cases}
J\colon u,&i=0,\\
J+(u),&i=1.
\end{cases}\]
Accordingly, for any $f\in\{0,1\}^{[n]}$ we set
\[J_f:=J\odot_{f(1)}x_1\odot\cdots\odot_{f(n)}x_n\]
and put $A_f:=f^{-1}(0)$ and $B_f:=f^{-1}(1)$.  Clearly,
\[J_f=\begin{cases}
(x_i\colon i\in B_f),&A_f\in\Delta(J),\\
(1),&\text{otherwise},
\end{cases}\]
that is $\Delta(J)=\{A_f\subseteq[n]\colon J_f\neq(1)\}$.
\begin{notation}
In what follows, we use the following notation:
\begin{itemize}
\item $S:=\KK[x_1,\ldots,x_n,y_1,\ldots,y_n]$ is a polynomial ring.
\item For subsets $V\subseteq[n]$, 
\[(x_V):=(x_i\colon i\in V)\quad\text{and}\quad(y_V):=(y_i\colon i\in V).\]
\item Let $T$ be a tree on $[n]$ and $\ell\colon V(T)\rightarrow[n]$ be an ascending labeling of $T$. Then we use the lexicographic ordering $\leq:=\leq_\ell$ induced by
\[x_{\ell(1)}>\cdots>x_{\ell(n)}>y_{\ell(1)}>\cdots>y_{\ell(n)}\]
without further reference. As in the previous section, we further assume that $v=\ell(v)$ for all $v\in V(T)$, for convenience.
\item The set of all paths of a forest $F$ through a vertex $u$ as a middle vertex is denoted by $\paths(F, u)$. For any path $P$ in $F$, let
\[g_P^*:=\begin{cases}
g_P,&|P|\ \text{is even},\\
g'_{P,\varnothing},&|P|\ \text{is odd}.\\
\end{cases}\]
Also, we define $\End(P)$ to be the set of end vertices of $P$ and $\End_x(P)$ to be the set of end vertices of $P$ labeled by $x$ in $g_P^*$. 
\item Let $J\subseteq S$ be a monomial. For any $f\in\{0,1\}^{[n]}$ and $k\in[n]$ we set
\[J_{k,f}:=J\odot_{f(1)}x_1\odot\cdots\odot_{f(k)}x_k\]
and put $A_{k,f}:=f^{-1}(0)\cap[k]$ and $B_{k,f}:=f^{-1}(1)\cap[k]$.
Also, for any $g\in\{0,1\}^{[n]}$ and $k\in[n]$ we set
\[J^{k,g}:=J\odot_{g(n)}y_n\odot\cdots\odot_{g(k)}y_k\]
and put $A'_{k,g}:=g^{-1}(0)\cap[n-k]^c$ and $B'_{k,g}:=g^{-1}(1)\cap[n-k]^c$. 
\item For any $f\in\{0,1\}^{[n]}$, let $\edges_f(T)$ be the set of all paths $u,v$ of length one such that $u\sim v\sim a$ for some $a\in A_{n,f}$ and $u,v>a$. Also, let $\paths'_f(T)$ be the set of all paths $P$ of $T$ satisfying the following conditions:
\begin{itemize}
\item[(1)]$\min(P)\in B_{n,f}$,
\item[(2)]$\min(P)\cap\End(P)=\varnothing$,
\item[(3)]$\End_x(P)\subseteq A_{n,f}$,
\item[(4)]$V(P)\setminus\End_2(P)\subseteq B_{n,f}$,
\end{itemize}
where by $\min(P)$ we mean the vertex of $P$ with minimum label, and $\End_i(P)$ ($i\geq0$) denotes the set of vertices of $P$ at distance at most $i$ from an end point of $P$. Put
\[\paths^*_f(T):=\edges_f(T)\cup\paths_f(T),\]
where
\[\paths_f(T):=\{P-\End_x(P)\colon P\in\paths'_f(T)\}\]
in which $\End_x(P)$ is the set of all (end) vertices $v$ of $P$ for which $x_v\mid \init_\leq(g^*_P)$. Finally, for any $V\subseteq [n]$, let $(y_V)^\vee$ be the ideal of $S$ generated by the product of $y_v$ for all $v\in V$.
\end{itemize}
\end{notation}

For a given ascending labeling $\ell$ of $T$, let $\Delta(T,\ell)$ be the simplicial complex on $[2n]$ for which $I_{\Delta(T,\ell)}=\init_\leq(L_T^\KK(2))$ as ideals of $S$. Since $S$ has two sets of indeterminates, namely $x_1,\ldots,x_n$ and $y_1,\ldots,y_n$ we may identify every element of $\Delta(T,\ell)$ as an ordered pair $(A,A')$, where $A,A'\subseteq[n]$ correspond to subsets of $\{x_1,\ldots,x_n\}$ and $\{y_1,\ldots,y_n\}$, respectively. Accordingly, 
\[\Delta(T,\ell)=\{(f^{-1}(0),g^{-1}(0))\colon f,g\in\{0,1\}^{[n]}\ \text{and}\ (I_{n,f})^{n,g}\neq(1)\}.\]
The following lemma helps us to describe $\Delta(T,\ell)$.
\begin{lemma}\label{I_{k,f} and (I_{n,f})^{k,g}}
Let $T$ be a tree on $[n]$ with an ascending labeling $\ell$ and $I=\init_\leq(L_T^\KK(2))$. If $f,g\in\{0,1\}^{[n]}$, then
\[(I_{n,f})^{n,g}=(x_{B_{n,f}})+(x_{N_T(A_{n,f})})^*+(y_{B'_{n,g}})+\sum_{P\in\paths^*_f(T)}(y_{V(P)\setminus A'_{n,g}})^\vee,\]
where 
\[(x_{N_T(A_{k,f})})^*=
\begin{cases}
(x_{N_T(A_{n,f})}),&A_{n,f}\emph{ is independent in }T,\\
(1),&\emph{otherwise}.
\end{cases}\]
\end{lemma}
\begin{proof}
First observe that if $J=\init_\leq(F)$ where $F$ is a subforest of $T$ and $u$ is the vertex of $F$ of minimum label, then
\begin{align}
J\colon x_u&=\init_\leq(F\setminus u)+(x_{N_F(u)})+\sum_{v\in N_F(u)}y_v(y_{N_F(v)\setminus u})+\sum_{P\in\paths(F, u)}\init_\leq(g_P^*)\label{J colon xu}\\
&=\init_\leq(F\setminus u)+(x_{N_F(u)})+\sum_{v\in N_F(u)}y_v(y_{N_F(v)\setminus u}),\label{J colon xu:simplified}\\
J+(x_u)&=\init_\leq(F\setminus u)+(x_u)+\sum_{P\in\paths(F, u)}\init_\leq(g_P^*).
\end{align}
In order to obtain \eqref{J colon xu:simplified} from \eqref{J colon xu} we note that every summand $\init_\leq(g_P^*)$ is either a multiple of $x_v$ or a multiple of $y_vy_w$ for some $v\in N_F(u)$ and $w\in N_F(u)\setminus\{u\}$.

If $f\in\{0,1\}^{[n]}$, then a simple inductive argument yields
\begin{align*}
I_{k,f}=&I\odot_{f(1)}x_1\odot\cdots\odot_{f(k)}x_k\\
=&\init_\leq(T\setminus[k])+(x_{B_{k,f}})+(x_{N_T(A_{k,f})})^*+\sum_{\substack{u\sim v\sim a\in A_{k,f}\\u,v>a}}(y_uy_v)\\
&+\sum_{\substack{i\in B_{k,f}\\P\in\paths(T,i)\\\End_x(P)\cap B_{k,f}=\varnothing}}\frac{\init_\leq(g_P^*)}{\prod_{j\in\End_x(P)\cap A_{k,f}}x_j},
\end{align*}
for all $k\in[n]$. In particular, for $k=n$, we get
\[I_{n,f}=(x_{B_{n,f}})+(x_{N_T(A_{n,f})})^*+\sum_{\substack{u\sim v\sim a\in A_{n,f}\\u,v>a}}(y_uy_v)+\sum_{\substack{i\in B_{n,f}\\P\in\paths(T,i)\\\End_x(P)\cap B_{n,f}=\varnothing}}\frac{\init_\leq(g_P^*)}{\prod_{j\in\End_x(P)\cap A_{n,f}}x_j}.\]

Let $P$ be a path in the second sum for which $\init_\leq(g^*_P)$ is not divisible by $y_uy_v$ for some $uv\in\edges_f(T)$. In the following, we show that $P$ satisfies the conditions (1)--(4) of the definition of $\paths'_f(T)$. From the unimodality of paths in $T$ it follows that $V(P)\setminus\End_2(P)\subseteq B_{n,f}$. We can further assume that $\min(P)\in B_{n,f}$ and $\min(P)\cap\End(P)=\varnothing$. If $A_{n,f}$ is not independent, then $I_{n,f}=(x_{N_T(A_{n,f})})^*=(1)$ and our assumption is superfluous. Suppose $A_{n,f}$ is an independent set in $T$ and $P$ is the path $v_1,\ldots,v_m$ with $v_1<v_m$ and $\min(P)=v_k\in A_{n,f}$. Notice that $v_1\in A_{n,f}$. Then $k\leq 3$ and $m-k\leq2$ otherwise $\init_\leq(g^*_P)$ is divisible by either $y_{v_{k-2}}y_{v_{k-1}}$ when $k>3$ or $y_{v_{k+1}}y_{v_{k+2}}$ when $m-k>2$ so that $P\notin\paths'_f(T)$. If $k=1$, then since $\init_\leq(g^*_P)$ is not divisible by $y_{v_2}y_{v_3}$ we must have $m=2$. Then $v_2\in A_{n,f}$ for $\End_x(P)\cap B_{n,f}=\varnothing$, which contradicts the fact that $A_{n,f}$ is independent. As $v_1\in A_{n,f}$ and $A_{n,f}$ is independent, it follows that $k=3$. If $m=5$, then $\init_\leq(g^*_P)$ is divisible by $y_{v_4}y_{v_5}$, a contradiction. Thus $m=4$ and $\init_\leq(g^*_P)=x_{v_1}y_{v_2}y_{v_3}x_{v_4}$ with $v_4\in A_{n,f}$ contradicting the fact that $A_{n,f}$ is independent. It turns out that 
\[I_{n,f}=(x_{B_{n,f}})+(x_{N_T(A_{n,f})})^*+\sum_{uv\in\edges_f(T)}(y_uy_v)+\sum_{P\in\paths'_f(T)}\frac{\init_\leq(g_P^*)}{\prod_{j\in\End_x(P)\cap A_{n,f}}x_j}.\]
Recall that 
\[\paths_f(T):=\{P-\End_x(P)\colon P\in\paths'_f(T)\}\]
and
\[\paths^*_f(T):=\edges_f(T)\cup\paths_f(T).\]
Next, we observe that
\begin{equation}\label{I_(n,f)^(k,g)}
\begin{split}
(I_{n,f})^{k,g}=&I_{n,f}\odot_{g(n)}y_n\odot\cdots\odot_{g(k)}y_k\\
=&(x_{B_{n,f}})+(x_{N_T(A_{n,f})})^*+(y_{B'_{n+1-k,g}})+\sum_{P\in\paths^*_f(T)}(y_{V(P)\setminus A'_{n+1-k,g}})^\vee
\end{split}
\end{equation}
for all $f,g\in\{0,1\}^{[n]}$ and $k\in[n]$. Recall that $(y_{V(P)\setminus A'_{n+1-k,g}})^\vee$ is the ideal of $S$ generated by products of all $y_v$ for $v\in V(P)\setminus A'_{n+1-k,g}$. Therefore, 
\[(I_{n,f})^{n,g}=(x_{B_{n,f}})+(x_{N_T(A_{n,f})})^*+(y_{B'_{n,g}})+\sum_{P\in\paths^*_f(T)}(y_{V(P)\setminus A'_{n,g}})^\vee,\]
as required.
\end{proof}

One observes from the statement of Lemma \ref{I_{k,f} and (I_{n,f})^{k,g}} that 
\begin{itemize}
\item[(1)]$(x_{N_T(A_{n,f})})^*\neq(1)$ if and only if $A_{n,f}$ is an independent set in $T$,
\item[(2)]$(y_{V(P)\setminus A'_{n+1-k,g}})^\vee=(1)$ for a path $P$ of $\paths^*_f(T)$ if and only if $V(P)\subseteq A'_{n,g}$.
\end{itemize}
It follows that $(I_{n,f})^{n,g}\neq(1)$ if and only if $(x_{N_T(A_{n,f})})^*\neq(1)$ and $(y_{V(P)\setminus A'_{n+1-k,g}})^\vee\neq(1)$ for all $P\in\paths^*_f(T)$. In this case, $N_T(A_{n,f})\subseteq B_{n,f}$, $\varnothing\neq V(P)\setminus A'_{n,g}\subseteq B'_{n,g}$ for all $P\in\paths^*_f(T)$, and consequently
\begin{align*}
(I_{n,f})^{n,g}&=(x_{B_{n,f}})+(x_{N_T(A_{n,f})})^*+(y_{B'_{n,g}})+\sum_{P\in\paths^*_f(T)}(y_{V(P)\setminus A'_{n,g}})^\vee\\
&=(x_{B_{n,f}})+(y_{B'_{n,g}}).
\end{align*}

Let $T$ be a tree on $[n]$ with a given ascending labeling $\ell$. In view of the above equalities, we are interested in those $f,g$ for which $(I_{n,f})^{n,g}\neq(1)$ in order to compute the Hilbert series of $S/L_T^\KK(2)$. The above arguments show that $\Delta(T,\ell)$ is actually the set of all ordered pairs $(A,A')$ of subsets of $[n]$ such that $A$ is independent in $T$ and $V(P)\nsubseteq A'$ for all $P\in\paths^*_A(T)$, where $\paths^*_A(T):=\paths^*_f(T)$ for $f\in\{0,1\}^{[n]}$ with $f^{-1}(0)=A$. 

According to \eqref{Hilb(R/I,t)}, \cite[Proposition 2.2.5]{jh-th}, and \cite[Theorem 5.1.4]{wb-jh}, we have the following proposition:
\begin{proposition}\label{Hilbert series}
Let $T$ be a tree on $[n]$ with an ascending labeling $\ell$. Then 
\[\Hilb\left(\frac{S}{L_T^\KK(2)}, t\right)=\sum_{(A,A')\in\Delta(T,\ell)}\left(\frac{t}{1-t}\right)^{|A|+|A'|}.\]
In particular,
\[\dim\left(\frac{S}{L_T^\KK(2)}\right)=\dim\Delta(T,\ell)+1=\max\left\{|A|+|A'|\colon\ (A,A')\in\Delta(T,\ell)\right\}.\]
\end{proposition}
\begin{corollary}\label{Lower and upper bounds for dim(S/I)}
If $T$ is a tree on $[n]$, then
\[n+1\leq\dim\left(\frac{S}{L_T^\KK(2)}\right)\leq p(T)+n-1,\]
where $p(T)$ denotes the number of pendants of $T$.
\end{corollary}
\begin{proof}
Let $P(T')$ stand for the set of all pendants of a tree $T'$ and put $p(T'):=|P(T')|$. Suppose without loss of generality that $n\geq3$. If $A=\{n\}$ and $A'=[n]$, then $(A,A')\in\Delta(T,\ell)$ for any ascending labeling $\ell$ of $T$. Thus $\dim(S/L_T^\KK(2))\geq n+1$ by Proposition \ref{Hilbert series}. To prove the upper bound, let $A$ be an independent set in $T$, and let $v\in P(T)$. We may view $T$ as a rooted tree with root $v$ and take a labeling $\ell\colon V(T)\rightarrow[n]$ of $T$ with $\ell(u)<\ell(u')$ whenever $d(u,v)<d(u',v)$. Clearly, $\ell$ is an ascending labeling of $T$. Let $(A,A')\in\Delta(T,\ell)$. Let $a\neq v$ be an element of $A$ that is neither a pendant of $T$ nor a pendant of $T\setminus (P(T)\setminus\{v\})$ that is
$a\in A\setminus(P(T)\setminus\{v\})\cup (P(T\setminus (P(T)\setminus\{v\}))\setminus\{v\})$. Let $f\in\{0,1\}^{[n]}$ be such that $f^{-1}(0)=A$. From the definition of $\Delta(T,\ell)$, it follows that either $s\notin A'$ or $t\notin A'$ for any children $s$ of $a$ and children $t$ of $s$, otherwise the path $s,t$ of $\edges_f(T)\subseteq\paths^*_f(T)$ is contained in $T[A']$, contradicting the fact that $(A,A')\in\Delta(T,\ell)$. Thus 
\[|A'|\leq n - |A\setminus((P(T)\setminus\{v\})\cup (P(T\setminus (P(T)\setminus\{v\}))\setminus\{v\}))|,\]
which implies that
\[|A|+|A'|\leq n + |A\cap((P(T)\setminus\{v\})\cup (P(T\setminus (P(T)\setminus\{v\}))\setminus\{v\}))|.\]
Since $A$ does not contain an element $u$ of $P(T\setminus (P(T)\setminus\{v\}))\setminus\{v\}$ and any of its children $u'$ in $P(T)\setminus\{v\}$ simultaneously, it follows that 
\[|A\cap((P(T)\setminus\{v\})\cup (P(T\setminus (P(T)\setminus\{v\}))\setminus\{v\}))|\leq|A\cap(P(T)\setminus\{v\})|\leq|P(T)\setminus\{v\}|=p(T)-1.\]
Therefore, $|A|+|A'|\leq p(T)+n-1$. 
\end{proof}
\begin{remark} \label{regular sequence}
Let $R$ be a standard graded $\KK$-algebra and $I$ be a graded ideal of $R$. If $u_1, \ldots, u_k$ are homogeneous elements of $R$ such that the sequence $u_1+I, \ldots, u_k+I$ forms an ${R}/{I}$-regular sequence, then the Hilbert series of $\frac{{R}}{(I,u_1,\ldots,u_k) } $ is of the form
\begin{equation*}
H_{\frac{{R}}{(I,u_1,\ldots,u_k)}}(t)=\prod_{i=1}^k \left(1-t^{\deg(u_i)}\right) \cdot H_{\frac{{R}}{I}}(t).
\end{equation*}
\end{remark}
\begin{remark}
Let $T$ be a tree on $[n]$ with an ascending labeling $\ell$. Let $d:=\dim\Delta(T,\ell)$ and
\[\Delta_i(T,\ell):=\{(A,A')\in\Delta(T,\ell)\colon |A|+|A'|=i\},\]
for all $0\leq i\leq d$. Then
\[\Hilb\left(\frac{S}{L_T^\KK(2)}, t\right)=\frac{1}{(1-t)^d}\sum_{i=0}^d|\Delta_i(T,\ell)|t^i(1-t)^{d-i}.\]
In particular, $|\Delta_i(T,\ell)|$ is independent of the choice of the labeling $\ell$, for all $i=0,\ldots,d$.
\end{remark}
\begin{example}
In this example, we compute $\Delta_i(T,\ell)$ when $T$ is a path or a star graph.
\begin{itemize}
\item[(1)]If $T=P_n$ is a path with $n$ vertices, then $S/L_T^\KK(2)$ is a complete intersection by \cite[Theorem 1.5(2)]{ac-vw}. As $\dim(S/L_T^\KK(2))=n+1$ (see Corollary \ref{Lower and upper bounds for dim(S/I)}), it follows that $L_T^\KK(2)$ is generated by a homogeneous regular sequence $f_1,\ldots,f_{n-1}$ of degree $2$. Thus in view of Remark~\ref{regular sequence}, we have
\begin{align*}
\Hilb\left(\frac{S}{L_T^\KK(2)}, t\right)&=(1-t^2)^{n-1}\Hilb\left(S, t\right)=\frac{(1+t)^{n-1}}{(1-t)^{n+1}}\\
&=\frac{1}{(1-t)^{n+1}}\sum_{i=0}^{n-1}\binom{n-1}{i}t^i.
\end{align*}
On the other hand, 
\[\Hilb\left(\frac{S}{L_T^\KK(2)}, t\right)=\frac{1}{(1-t)^{n+1}}\sum_{i\geq0}|\Delta_i(T,\ell)|t^i(1-t)^{n+1-i},\]
for a given ascending labeling $\ell$ of $T$ from which it follows that
\[|\Delta_i(T,\ell)|=\sum_{k=0}^i\binom{n+1-k}{i-k}\binom{n-1}{k}.\]
for all $i$ with $0\leq i\leq n+1$. 
\item[(2)]Suppose $T=S_n$ is the star graph with $n$ vertices and let $\ell$ be an ascending labeling of $T$ that assigns $1$ to the central vertex. Then $(A,A')\in\Delta_i(T,\ell)$ if and only if one of the following cases holds:
\begin{itemize}
\item[(a)]$A=\varnothing$ and $A'\subseteq[n]$ with $|A'|=i$,
\item[(b)]$A=\{1\}$ and $A'\subseteq[n]$ with $|A|=i-1$,
\item[(c)]$1\notin A\neq\varnothing$ and either
\begin{itemize}
\item[(c1)]$A'\subseteq[n]\setminus\{1\}$ with $|A'|=i-|A|$, or
\item[(c2)]$1\in A'$ and $A'\subseteq[\min(A)]$ with $|A'|=i-|A|$,
\end{itemize}
\end{itemize}
The number of $(A,A')$ in cases (a) and (b) are simply $\binom{n}{i}$ and $\binom{n}{i-1}$. Also, this number for case (c.c1) is equal to $\sum_{k=1}^i\binom{n-1}{k}\binom{n-1}{i-k}$, where $\binom{n-1}{k}$ and $\binom{n-1}{i-k}$ are the number of subsets $A$ and $A'$ of $[n]\setminus\{1\}$ of size $k$ and $i-k$, respectively. To count the number of $(A,A')$ in case (c.c2), notice that either $A\cap A'=\varnothing$ and $A\cup A'=\{1=t_1<\cdots<t_i\}$ is a subset of $[n]$ including $1$ where $A'=\{t_1,\ldots,t_k\}$ and $A=\{t_{k+1},\ldots,t_i\}$ for $1\leq k<i$, or $|A\cap A'|=1$ and $A\cup A'=\{1=t_1<\cdots<t_{i-1}\}$ is a subset of $[n]$ including $1$ where $A'=\{t_1,\ldots,t_k\}$ and $A=\{t_k,\ldots,t_{i-1}\}$ for $1<k<i$. In the former case, we have $(i-1)\binom{n-1}{i-1}$ pairs $(A,A')$, and in the later case, we have $(i-2)\binom{n-2}{i-2}$ pairs $(A,A')$. Thus
\begin{align*}
|\Delta_i(T,\ell)|=&\binom{n}{i}+\binom{n}{i-1}+\sum_{k=1}^i\binom{n-1}{k}\binom{n-1}{i-k}\\
&+\left[(i-1)\binom{n-1}{i-1}+(i-2)\binom{n-1}{i-2}\right]\\
=&\binom{n}{i}+\binom{n}{i-1}+\binom{2n-2}{i}-\binom{n-1}{i}+(n-1)\binom{n-1}{i-2}\\
=&\binom{2n-2}{i}+(n-1)\binom{n-1}{i-2}+\binom{n}{i-1}+\binom{n-1}{i-1}
\end{align*}
as
\[\sum_{k=0}^i\binom{m}{k}\binom{m}{i-k}=\binom{2m}{i}\]
for all $0\leq i\leq m$.
\end{itemize}
\end{example}


In \cite{jh-am-ssm-vw}, Herzog, Macchia, Saeedi Madani, and Welker present a primary decomposition for $L_G^\KK(2)$ when $\sqrt{-1}\notin\KK$ and apply it to give a formula for the dimension of $S/L_G^\KK(2)$ in this case. Indeed, they show that if $\Gamma$ is a graph on $[n]$, $I=L_\Gamma^\KK(2)$, and $\sqrt{-1}\notin\KK$, then
\[\dim\left(\frac{S}{I}\right)=\max\left\{|V|+b(\Gamma[V])\colon\ V\subseteq[n]\right\},\]
where $b(\Gamma)$ denotes the number of bipartite connected components of a graph $\Gamma$. It follows that, if $T$ is a tree, $I=L_T^\KK(2)$, and $\sqrt{-1}\notin\KK$, then
\begin{equation}\label{dim(S/I)}
\dim\left(\frac{S}{I}\right)=\max\left\{|V|+c(T[V])\colon\ V\subseteq[n]\right\},
\end{equation}
where $c(\Gamma)$ denotes the number of connected components of a graph $\Gamma$. According to Proposition \ref{Hilbert series}, the Hilbert series of $S/I$ is independent of the underlying field $\KK$. Thus the formula \eqref{dim(S/I)} for the dimension of $S/I$ is valid for any field $\KK$. In Corollary \ref{dim(S/I): Explicit formula} below, we give an explicit formula for the dimension of the ring $S/I$. 

We say that a path $P$ in a tree $T$ is \textit{pendant} if it contains a pendant vertex and that all of its vertices have at most two neighbors in $T$. For a tree $T$, we put $\P_0^*(T)=\P_0(T)=\Q_0(T)=\varnothing$, and define $\P_i^*(T),\P_i(T),\Q_i(T)$ inductively as follows for all $i\geq1$: if $T_i:=T-\cup_{j=0}^{i-1}(\P_j(T)\cup \Q_j(T))$ and $\P_i^*(T)$ is the set of all maximal pendant paths of $T_i$, then $\P_i(T):=V(\cup\P_i^*(T))$ and $\Q_i(T):=N_{T_i}(\P_i(T))\setminus\P_i(T)$. Also, we put $\P^*(T):=\cup_i\P_i^*(T)$, $\P(T):=\cup_i\P_i(T)$, and $\Q(T):=\cup_i\Q_i(T)$.
\begin{example}
Let $T$ be the tree in Fig. \ref{P, Q}. Then $\P_1(T)$, $\P_2(T)$, $\Q_1(T)$, and $\Q_2(T)$ are the sets of all black square vertices, black circle vertices, white square vertices, and white circle vertices, respectively. Also, $\P^*(T)$ is the set of all black-vertex paths.
\end{example}
\begin{figure}[ht]
\begin{tikzpicture}
\node [draw, circle, inner sep = 2pt] (0) at (0, 0) {};
\node [draw, circle, inner sep = 2pt] (1) at (1, 0) {};

\draw (0)--(1);

\node [draw, circle, fill=black, inner sep = 2pt] (10) at ({1 + cos(-40)}, {sin(-40)}) {};
\node [draw, circle, fill=black, inner sep = 2pt] (11) at ({1 + cos(40)}, {sin(40)}) {};

\node [draw, rectangle, inner sep = 2pt] (100) at ({1 + cos(-40)+cos(-60)}, {sin(-40)+sin(-60)}) {};
\node [draw, rectangle, inner sep = 2pt] (101) at ({1 + cos(-40)+cos(0)}, {sin(-40)+sin(0)}) {};

\node [draw, rectangle, inner sep = 2pt] (110) at ({1 + cos(40)+cos(0)}, {sin(40)+sin(0)}) {};
\node [draw, rectangle, inner sep = 2pt] (111) at ({1 + cos(40)+cos(60)}, {sin(40)+sin(60)}) {};

\node [draw, rectangle, fill=black, inner sep = 2pt] (1000) at ({1 + cos(-40)+cos(-60)+cos(-90)}, {sin(-40)+sin(-60)+sin(-90)}) {};
\node [draw, rectangle, fill=black, inner sep = 2pt] (1001) at ({1 + cos(-40)+cos(-60)+cos(-30)}, {sin(-40)+sin(-60)+sin(-30)}) {};

\node [draw, rectangle, fill=black, inner sep = 2pt] (1010) at ({1 + cos(-40)+cos(0)+cos(-30)}, {sin(-40)+sin(0)+sin(-30)}) {};
\node [draw, rectangle, fill=black, inner sep = 2pt] (1011) at ({1 + cos(-40)+cos(0)+cos(30)}, {sin(-40)+sin(0)+sin(30)}) {};

\node [draw, rectangle, fill=black, inner sep = 2pt] (1100) at ({1 + cos(40)+cos(0)+cos(-30)}, {sin(40)+sin(0)+sin(-30)}) {};
\node [draw, rectangle, fill=black, inner sep = 2pt] (1101) at ({1 + cos(40)+cos(0)+cos(30)}, {sin(40)+sin(0)+sin(30)}) {};

\node [draw, rectangle, fill=black, inner sep = 2pt] (1110) at ({1 + cos(40)+cos(60)+cos(30)}, {sin(40)+sin(60)+sin(30)}) {};
\node [draw, rectangle, fill=black, inner sep = 2pt] (1111) at ({1 + cos(40)+cos(60)+cos(90)}, {sin(40)+sin(60)+sin(90)}) {};

\node [draw, rectangle, fill=black, inner sep = 2pt] (1101a) at ({1 + cos(40)+cos(0)+cos(30)+1}, {sin(40)+sin(0)+sin(30)}) {};
\node [draw, rectangle, fill=black, inner sep = 2pt] (1101b) at ({1 + cos(40)+cos(0)+cos(30)+2}, {sin(40)+sin(0)+sin(30)}) {};
\node [draw, rectangle, fill=black, inner sep = 2pt] (1101c) at ({1 + cos(40)+cos(0)+cos(30)+3}, {sin(40)+sin(0)+sin(30)}) {};

\draw (10)--(1)--(11);
\draw (100)--(10)--(101);
\draw (110)--(11)--(111);
\draw (1000)--(100)--(1001);
\draw (1010)--(101)--(1011);
\draw (1100)--(110)--(1101);
\draw (1110)--(111)--(1111);

\draw (1101)--(1101a) (1101b)--(1101c);
\draw [dashed] (1101a)--(1101b);

\node [draw, circle, fill=black, inner sep = 2pt] (00) at ({-cos(-40)}, {sin(-40)}) {};
\node [draw, circle, fill=black, inner sep = 2pt] (01) at ({-cos(40)}, {sin(40)}) {};

\node [draw, rectangle, inner sep = 2pt] (000) at ({-cos(-40)-cos(-60)}, {sin(-40)+sin(-60)}) {};
\node [draw, rectangle, inner sep = 2pt] (001) at ({-cos(-40)-cos(0)}, {sin(-40)+sin(0)}) {};

\node [draw, rectangle, inner sep = 2pt] (010) at ({-cos(40)-cos(0)}, {sin(40)+sin(0)}) {};
\node [draw, rectangle, inner sep = 2pt] (011) at ({-cos(40)-cos(60)}, {sin(40)+sin(60)}) {};

\node [draw, rectangle, fill=black, inner sep = 2pt] (0000) at ({-cos(-40)-cos(-60)-cos(-90)}, {sin(-40)+sin(-60)+sin(-90)}) {};
\node [draw, rectangle, fill=black, inner sep = 2pt] (0001) at ({-cos(-40)-cos(-60)-cos(-30)}, {sin(-40)+sin(-60)+sin(-30)}) {};

\node [draw, rectangle, fill=black, inner sep = 2pt] (0010) at ({-cos(-40)-cos(0)-cos(-30)}, {sin(-40)+sin(0)+sin(-30)}) {};
\node [draw, rectangle, fill=black, inner sep = 2pt] (0011) at ({-cos(-40)-cos(0)-cos(30)}, {sin(-40)+sin(0)+sin(30)}) {};

\node [draw, rectangle, fill=black, inner sep = 2pt] (0100) at ({-cos(40)-cos(0)-cos(-30)}, {sin(40)+sin(0)+sin(-30)}) {};
\node [draw, rectangle, fill=black, inner sep = 2pt] (0101) at ({-cos(40)-cos(0)-cos(30)}, {sin(40)+sin(0)+sin(30)}) {};

\node [draw, rectangle, fill=black, inner sep = 2pt] (0110) at ({-cos(40)-cos(60)-cos(30)}, {sin(40)+sin(60)+sin(30)}) {};
\node [draw, rectangle, fill=black, inner sep = 2pt] (0111) at ({-cos(40)-cos(60)-cos(90)}, {sin(40)+sin(60)+sin(90)}) {};

\node [draw, rectangle, fill=black, inner sep = 2pt] (0010a) at ({-cos(-40)-cos(0)-cos(-30)-1}, {sin(-40)+sin(0)+sin(-30)}) {};
\node [draw, rectangle, fill=black, inner sep = 2pt] (0010b) at ({-cos(-40)-cos(0)-cos(-30)-2}, {sin(-40)+sin(0)+sin(-30)}) {};
\node [draw, rectangle, fill=black, inner sep = 2pt] (0010c) at ({-cos(-40)-cos(0)-cos(-30)-3}, {sin(-40)+sin(0)+sin(-30)}) {};

\draw (00)--(0)--(01);
\draw (000)--(00)--(001);
\draw (010)--(01)--(011);
\draw (0000)--(000)--(0001);
\draw (0010)--(001)--(0011);
\draw (0100)--(010)--(0101);
\draw (0110)--(011)--(0111);

\draw (0010)--(0010a) (0010b)--(0010c);
\draw [dashed] (0010a)--(0010b);
\end{tikzpicture}
\caption{}\label{P, Q}
\end{figure}
\begin{proposition}\label{dim(S/I): Explicit formula}
If $T$ is a tree on $[n]$ and $I=L_T^\KK(2)$, then
\[\dim\left(\frac{S}{I}\right)=|\P(T)|+|\P^*(T)|.\]
\end{proposition}
\begin{proof}
Let $d(V):=|V|+c(T[V])$ for any subset $V$ of $[n]$. We show that 
\begin{equation}\label{d(P(T) cup Q)}
d(\P(T)\cup Q)=d(\P(T))+2|Q|-\sum_{q\in Q}\deg_{T[\P]}(q)-|E(T[Q])|
\end{equation}
for every subset $Q$ of $\Q(T)$. Clearly, the equality \eqref{d(P(T) cup Q)} holds for $Q=\varnothing$. Suppose we have proved the equality \eqref{d(P(T) cup Q)} for some subset $Q$ of $\Q(T)$ and $q\in \Q(T)\setminus Q$. Assume $q\in\Q_i(T)$. Then
\begin{align*}
d(\P(T)\cup Q\cup\{q\})&=d(\P(T)\cup Q)+1-(\deg_{T[\P(T)]}(q)-1)\\
&=d(\P(T))+2|Q\cup\{q\}|-\sum_{q'\in Q\cup\{q\}}\deg_{T[\P]}(q')-|E(T[Q\cup\{q\}])|
\end{align*}
for all vertices in $\P(T)\cup Q$ which are adjacent to $q$ in $T$ belong to distinct connected components of $T[\P(T)\cup Q]$ and that these components turn into a single connected component when $q$ is added to $\P(T)\cup Q$. It follows that the equality \eqref{d(P(T) cup Q)} holds for $Q\cup\{q\}$ as well. Next, we show that
\begin{equation}\label{d((P(T) - P) cup Q)}
d((\P(T)\setminus P)\cup Q)=d(\P(T)\cup Q) - 2|P| + |N_{T[\P(T)]}(P)|+\sum_{p\in P}\deg_{T[Q]}(p)
\end{equation}
for all subsets $P$ of $\P(T)$ and $Q$ of $\Q(T)$. Let $Q$ be a fixed subset of $\Q(T)$. Clearly, the equality \eqref{d((P(T) - P) cup Q)} holds for $P=\varnothing$. Assume the equality \eqref{d((P(T) - P) cup Q)} is true for a subset $P$ of $\P(T)$ and that $p\in\P(T)\setminus P$. We further assume that $p\in\calP$ for some $\calP\in\P_i^*(T)$. We have the following cases to consider:
\begin{itemize}
\item[(i)]$p\not\in\End(\calP)$. Then either $N_\calP(p)\cap P=\varnothing$ for which \begin{align*}
d((\P(T)\setminus (P\cup\{p\}))\cup Q)&=d((\P(T)\setminus (P\cup\{p\}))\cup Q)\\
&=d((\P(T)\setminus (P\cup\{p\}))\cup Q)-2+\deg_{T[\P(T)\cup Q]}(p)
\end{align*}
or $N_\calP(p)\cap P\neq\varnothing$ so that
\begin{align*}
d((\P(T)\setminus (P\cup\{p\}))\cup Q)&=d((\P(T)\setminus (P\cup\{p\}))\cup Q)-1\\
&=d((\P(T)\setminus (P\cup\{p\}))\cup Q)-2+\deg_{T[\P(T)\cup Q]}(p).
\end{align*}
\item[(ii)]$p$ is adjacent to some vertices of $\Q_{i-1}(T)$, say $q_1,\ldots,q_k$. As $q_1,\ldots,q_k$ belong to distinct connected components of $T[(\P(T)\setminus (P\cup\{p\}))\cup Q]$, it follows that
\begin{align*}
d((\P(T)\setminus (P\cup\{p\}))\cup Q)&=d((\P(T)\setminus (P\cup\{p\}))\cup Q)-1+k-\deg_{T[P\cup\{p\}]}(p)\\
&=d((\P(T)\setminus (P\cup\{p\}))\cup Q)-2+\deg_{T[\P(T)\cup Q]}(p).
\end{align*}
\item[(iii)]$p$ is adjacent to a vertex of $\Q_i(T)$ that is unique by our construction of $\P_i(T)$. Then
\begin{align*}
d((\P(T)\setminus (P\cup\{p\}))\cup Q)&=d((\P(T)\setminus (P\cup\{p\}))\cup Q)-1+\deg_{T[P\cup\{p\}]}(p)\\
&=d((\P(T)\setminus (P\cup\{p\}))\cup Q)-2+\deg_{T[\P(T)\cup Q]}(p).
\end{align*}
\end{itemize}
The above arguments yield
\begin{align*}
d((\P(T)\setminus (P\cup\{p\}))\cup Q)=&d((\P(T)\setminus (P\cup\{p\}))\cup Q)-2+\deg_{T[\P(T)\cup Q]}(p)\\
=&d(\P(T)\cup Q) - 2|P\cup\{p\}| + |N_{T[\P(T)]}(P\cup\{p\})|\\
&+\sum_{p'\in P\cup\{p\}}\deg_{T[Q]}(p')
\end{align*}
that is the equality \eqref{d((P(T) - P) cup Q)} holds for $P\cup\{p\}$ as well.

Now, from the equalities \eqref{d(P(T) cup Q)} and \eqref{d((P(T) - P) cup Q)} for arbitrary subsets $P$ of $\P(T)$ and $Q$ of $\Q(T)$, we get
\begin{align*}
d((\P(T)\setminus P)\cup Q)=&d(\P(T)\cup Q)-2|P|+|N_{T[\P(T)]}(P)|+\sum_{p\in P}\deg_{T[Q\cup\{p\}]}(p)\\
=&d(\P(T))-2|P|+|N_{T[\P(T)]}(P)|+2|Q|+\sum_{p\in P}\deg_{T[Q\cup\{p\}]}(p)\\
&-\sum_{q\in Q}\deg_{T[\P(T)\cup\{q\}]}(q)-|E(T[Q])|\\
=&d(\P(T))-2|P|+|N_{T[\P(T)]}(P)|+2|Q|+\sum_{p\in P}\deg_{T[Q\cup\{p\}]}(p)\\
&-\sum_{q\in Q}\deg_{T[P\cup\{q\}]}(q)-\sum_{q\in Q}\deg_{T[(\P(T)\setminus P)\cup\{q\}]}(q)-|E(T[Q])|\\
=&d(\P(T))-2|P|+|N_{T[\P(T)]}(P)|+2|Q|\\
&-\sum_{q\in Q}\deg_{T[(\P(T)\setminus P)\cup\{q\}]}(q)-|E(T[Q])|\\
\leq &d(\P(T))-2|P|+|N_{T[\P(T)]}(P)|+2|Q|-\sum_{q\in Q}\deg_{T[(\P(T)\setminus P)\cup\{q\}]}(q).
\end{align*}
Let $P_1$ be the set of all vertices in $P$ that are not end-points of some paths in $\P^*(T)$, and $P_2:=P\setminus P_1$. One observes that,
\begin{align*}
-2|P|+|N_{T[\P(T)]}(P)|&\leq(-2|P_1|+|N_{T[\P(T)]}(P_1)|)+(-2|P_2|+|N_{T[\P(T)]}(P_2)|)\\
&\leq -2|P_2|+|N_{T[\P(T)]}(P_2)|\\
&\leq -|P_2|.
\end{align*}
Also, let $Q_1$ be the set of vertices in $Q$ adjacent to at least two vertices in $\P(T)\setminus P$, and put $Q_2:=Q\setminus Q_1$. Then
\begin{align*}
2|Q|-\sum_{q\in Q}\deg_{T[(\P(T)\setminus P)\cup\{q\}]}(q)=&(2|Q_1|-\sum_{q_1\in Q_1}\deg_{T[(\P(T)\setminus P)\cup\{q_1\}]}(q_1))\\
&+(2|Q_2|-\sum_{q_2\in Q_2}\deg_{T[(\P(T)\setminus P)\cup\{q_2\}]}(q_2))\\
\leq& 2|Q_2|-\sum_{q_2\in Q_2}\deg_{T[(\P(T)\setminus P)\cup\{q_2\}]}(q_2)\\
=& 2|Q_2|-\sum_{q_2\in Q_2}\deg_{T[(\P(T)\setminus P_2)\cup\{q_2\}]}(q_2).
\end{align*}
Thus
\[d((\P(T)\setminus P)\cup Q)\leq d(\P(T))-|P_2|+2|Q_2|-\sum_{q_2\in Q_2}\deg_{T[(\P(T)\setminus P_2)\cup\{q_2\}]}(q_2).\]
From the construction of $\P(T)$ and $\Q(T)$, we know that 
\begin{equation}\label{deg_P(T)(q)}
\deg_{T[\P_i(T)\cup\P_{i+1}(T)\cup\{q\}]}(q)\geq3
\end{equation}
for all $q\in\Q_i(T)$. Let $m$ be the smallest positive integer for which $\P_{m+1}(T)=\varnothing$. Also, let $F$ be the subgraph of $T$ induced by edges between $P_2$ and $Q_2$. Since $F'$ is a forest, we observes that
\begin{align*}
|P_2|+|Q_2|=|V(F')|&=|E(F')|+c(F')\\
&=\sum_{q_2\in Q_2}\deg_{F'}(q_2)+c(F')\\
&\geq\sum_{q_2\in Q_2}\deg_{T[P_2\cup\{q_2\}]}(q_2).
\end{align*}
Then
\begin{align*}
d((\P(T)\setminus P)\cup Q)\leq& d(\P(T))-|P_2|+2|Q_2|-\sum_{q_2\in Q_2}\deg_{T[(\P(T)\setminus P_2)\cup\{q_2\}]}(q_2)\\
\leq& d(\P(T))+3|Q_2|-\sum_{q_2\in Q_2}\deg_{T[P_2\cup\{q_2\}]}(q_2)\\
&-\sum_{q_2\in Q_2}\deg_{T[(\P(T)\setminus P_2)\cup\{q_2\}]}(q_2)\\
\leq& d(\P(T))+3|Q_2|-\sum_{q_2\in Q_2}\deg_{T[\P(T)\cup\{q_2\}]}(q_2)\\
\leq& d(\P(T))
\end{align*}
for $\deg_{T[\P(T)\cup\{q_2\}]}(q_2)\geq3$ for all $q_2\in Q_2$ (see \eqref{deg_P(T)(q)}). Therefore,
\[\dim(S/I)=\max\{d((\P(T)\setminus P)\cup Q)\colon\ P\subseteq\P(T),\ Q\subseteq\Q(T)\}=d(\P(T)),\]
as required.
\end{proof}
\begin{remark}
Let $T$ be a tree with an ascending labeling $\ell$. Let $A_i(T)$ be the set of all pendants of paths in $\P_i^*(T)$ with maximum labels, and $A'_i=\P_i(T)$, for all $i\geq1$. If $A=\cup A_i$ and $A'=\cup A'_i$, then $(A,A')\in\Delta(T,\ell)$ and $|A|+|A'|=|\P(T)|+|\P^*(T)|$. Thus
\[\dim(S/I)=|A|+|A'|.\]
\end{remark}
\begin{example}
\begin{itemize}
\item[(1)]If $I=L_{P_n}^{\KK}(2)$, then $\dim(S/I)=n+1$ for $\P(P_n)=V(P_n)$ and $\P^*(P_n)=\{P_n\}$.
\item[(2)]If $I=L_{S_n}^{\KK}(2)$, then $\dim(S/I)=2n-2$ for $\P(S_n)=[n-1]$ and $\P^*(S_n)$ is the set of all pendants of $S_n$ as induced subgraphs of $S_n$ if $n$ is the central vertex of $S_n$.
\item[(3)]If $I=L_{T}^{\KK}(2)$ with $T$ being the tree in Fig. \ref{P, Q} with $n$ vertices, then $\dim(S/I)=n+10$ for $|\P(T)|=n-10$ and $|\P^*(T)|=20$.
\end{itemize}
\end{example}


\begin{thebibliography}{99}
\bibitem{wwa-pl}
W. W. Adams and P. Loustaunau, \textit{An Introduction to Gr\"obner Bases}, Graduate Studies in Mathematics, Vol. \textbf{3}, American Mathematical Society, 289 p., 1994.

\bibitem{aa-jh-th}
A. Aramova, J. Herzog, and T. Hibi, Gotzman theorems for exterior algebras and combinatorics, \textit{J. Algebra} \textbf{191} (1997), 174--211.

\bibitem{jab-usrm}
J. A. Bondy and U. S. R. Murty, \textit{Graph Theory with Applications}, American Elsevier, New York, 1976. 

\bibitem{wb-jh}
W. Bruns and J. Herzog, \textit{Cohen-Macaulay Rings}, Revised Edition, Cambridge University Press, Cambridge, 1996.

\bibitem{ac-vw}
A. Conca and V. Welker, Lov\'{a}sz-Saks-Schrijver ideals and coordinate sections of determinantal varieties, \textit{Algebra \& Number Theory} \textbf{13}(\textbf{2}) (2019), 455--484.

\bibitem{jh-am-ssm-vw}
J. Herzog, A. Macchia, S. Saeedi Madani, and V. Welker, On the ideal of orthogonal representations of a graph in $\mathbb{R}^2$, \textit{Adv. Appl. Math.} \textbf{71} (2015), 146--173.

\bibitem{jh-th}
J. Herzog, T. Hibi, \textit{Monomial Ideals}, in: GTM  \textbf{260}, Springer, London, 2011.

\bibitem{ak}
A. Kumar, Lov\'asz-Saks-Schrijver ideals and parity binomial edge ideals of graphs, \textit{Eur. J. Comb.} \textbf{93} (2021), 103274.

\bibitem{ll1}
L. Lov\'{a}sz, \textit{On the Shannon capacity of a graph}, IEEE Trans. Inform. Theory \textbf{25} (1979), 1--7.

\bibitem{ll2}
L. Lov\'{a}sz, \textit{Geometric representations of graphs}, preprint, 2009. 


\bibitem{ll-ms-as-89}
L. Lov\'{a}sz, M. Saks, and A. Schrijver, Orthogonal representations and connectivity of graphs, \textit{Linear Algebra Appl.} \textbf{114/115} (1989), 439--454.

\bibitem{ll-ms-as-00}
L. Lov\'{a}sz, M. Saks, and A. Schrijver, A correction: Orthogonal representations and connectivity of graphs, \textit{Linear Algebra Appl.} \textbf{313} (2000), 101--105.

\end{thebibliography}
\end{document}